\documentclass[a4paper,10pt]{article}
\usepackage[left=3cm,right=3cm,top=3cm,bottom=3.5cm]{geometry}
\usepackage[latin1]{inputenc}
\usepackage{latexsym}
\usepackage{amscd}
\usepackage{amsthm}
\usepackage{amsfonts}
\usepackage{graphicx}
\usepackage{color}
\usepackage{multicol}
\usepackage{enumerate}
\usepackage{makeidx}
\usepackage{indentfirst}
\usepackage{subfigure}
\usepackage{a4wide}
\usepackage{t1enc}
\usepackage{amsmath,amssymb}
\usepackage{ae,aecompl}
\usepackage{booktabs}
\usepackage{verbatim}
\usepackage{lipsum}
\usepackage{capt-of}
\usepackage{authblk}

\usepackage[font=footnotesize,labelfont=bf]{caption}
\usepackage[pagewise]{lineno}
\newcommand{\M}{\text{M}}
\newcommand{\m}{\text{m}}

\newcommand{\T}{\text{T}}

\newcommand{\NH}{\text{NHIM}}

\newcommand{\zeroM}{\text{0}}
\newcommand{\oneM}{\text{2}}
\newcommand{\zerom}{\text{1}}
\makeatletter
\newtheorem*{rep@theorem}{\rep@title}
\newcommand{\newreptheorem}[2]{%
\newenvironment{rep#1}[1]{%
 \def\rep@title{#2 \ref{##1}}%
 \begin{rep@theorem}}%
 {\end{rep@theorem}}}
\makeatother
\newtheorem{theorem}{Theorem}
\newtheorem{proposition}[theorem]{Proposition}

\newtheorem{lemma}[theorem]{Lemma}
\newreptheorem{proposition}{Proposition}
\newreptheorem{theorem}{Theorem}
\theoremstyle{definition}
\newtheorem{remark}[theorem]{Remark}
\newtheorem{definition}[theorem]{Definition}
\begin{document}
%%running title: Arnold diffusion with two harmonics  
\title{Arnold diffusion for a complete family of perturbations with two independent
      harmonics\footnote{This work has been partially supported by
      the Spanish MINECO-FEDER grant MTM2015-65715 and
      the Catalan grant 2014SGR504. AD has been also partially supported by
      the Russian Scientific Foundation grant 14-41-00044
      at the Lobachevsky University of Nizhny Novgorod. RS has been also partially supported by CNPq, Conselho Nacional de Desenvolvimento Cient\'{i}fico e Tecnol\'{o}gico - Brasil.}}
\author{Amadeu Delshams\thanks{amadeu.delshams@upc.edu}}
\author{Rodrigo G. Schaefer\thanks{rodrigo.schaefer@upc.edu}}
\affil{Departament de Matem\`atiques and Lab of Geometry and Dynamical Systems\\
   Universitat Polit\`ecnica de Catalunya, Barcelona}
   
\maketitle
\centerline{\emph{To Rafael de la Llave on the occasion of his 60th birthday}}

\begin{abstract}
We prove that for any non-trivial perturbation depending on any two independent harmonics of a pendulum and a rotor there is global instability.
The proof is based on the geometrical method and relies on
the concrete computation of several scattering maps. A complete description of the
different kinds of scattering maps taking place as well as the existence of piecewise smooth global scattering maps
is also provided.
\par\vspace{12pt}
\noindent MSC2010 numbers: 37J40

\noindent\emph{Keywords}:
  Arnold diffusion,
  Normally hyperbolic invariant manifolds,
  Scattering maps

 %codes of the Mathematics Subject Classification 2010 (MSC): 37J40

\end{abstract}

\section{Introduction}
\subsection{Main result}

We consider an \emph{a priori unstable} Hamiltonian with $2+1/2$ degrees of freedom
\begin{equation}
H_{\varepsilon}(p , q , I , \varphi , s) = \pm\left( \frac{p^2}{2} + \cos q  -1 \right) + \frac{I^2}{2} + \varepsilon h(q,\varphi,s)
\label{eq:hamil_system}
\end{equation}
consisting of a pendulum and a rotor plus a time periodic perturbation depending on two harmonics in the variables $(\varphi,s)$:
\begin{equation}
\begin{gathered}
h(q,\varphi,s)=f(q)g(\varphi, s),\\
f(q) = \cos q, \qquad g(\varphi , s) = a_1 \cos(k_1\varphi + l_1 s) + a_2\cos(k_2\varphi + l_2 s),\label{eq:g_general_case}
\end{gathered}
\end{equation}
with $k_1,\,k_2,\, l_1,\,l_2\in\mathbb{Z}$.

The goal of this paper is to prove that for \emph{any} non-trivial perturbation $a_1a_2 \neq 0$ depending on \emph{any} two \emph{independent} harmonics
$\arraycolsep=0.8pt\def\arraystretch{0.8}\left|\begin{array}{ll}k_1 & k_2 \\ l_1 & l_2\end{array}\right|\neq 0$,
there is global instability of the action $I$ for any $\varepsilon>0$ small enough.
\begin{theorem}\label{theo:main_theo}
Assume that $a_1a_2\neq 0$ and $k_1 l_2-k_2 l_1\neq 0$ in Hamiltonian \text{\eqref{eq:hamil_system}-\eqref{eq:g_general_case}}.
Then, for any $I^* > 0$, there exists $\varepsilon^* = \varepsilon^*(I^*, a_1, a_2)>0$ such that for any $\varepsilon$, $0<\varepsilon<\varepsilon^*$,
there exists a trajectory $\left(p(t),q(t),I(t), \varphi(t)\right)$ such that for some $T>0$
\begin{equation*}
I(0)\leq - I^*<I^*\leq I(T).
\end{equation*}
\end{theorem}

\begin{remark}
For a rough estimate of $\varepsilon^* \sim \exp( - \pi I^* /2 ) $ at least for $\left|a_1/a_2\right| < 0.625$, $k_1 = l_2 = 1$ and $l_1 = k_2 = 0$, and of the diffussion time $T = T(\varepsilon^* , I^* , a_1 , a_2) \sim (\T_s(I^* , a_1 , a_2)/\varepsilon)\log( C(I^* , a_1 , a_2) /\varepsilon)$ the reader is referred to \cite{Delshams2017}.
Analogous estimates could be obtained for all the other values of the parameters. 
\end{remark}

The proof is based on the geometrical method introduced in~\cite{Seara2006} and relies on
the concrete computation of several \emph{scattering maps}. A scattering map is a map
of transverse homoclinic orbits to a \emph{normally hyperbolic invariant manifold} (NHIM).
For Hamiltonian~\eqref{eq:hamil_system}, the NHIM turns out to be simply
\begin{equation}
\tilde{\Lambda}_{\varepsilon}=\tilde{\Lambda} = \left\{( 0 , 0 , I , \varphi , s ) : ( I , \varphi , s ) \in \mathbb{R}\times \mathbb{T}^2\right\}.\label{eq:NHIM_manifold}
\end{equation}
In the unperturbed case, i.e., $\varepsilon = 0$, for any $I^* >0$ the NHIM $\tilde{\Lambda}$ possesses a 4D \emph{separatrix},
that is to say, coincident stable and unstable invariant manifolds
\begin{equation*}
W^0\tilde{\Lambda} = \left\{ ( p_0( \tau ) , q_0( \tau ) , I , \varphi , s ) : \tau\in\mathbb{R} , I\in\left[-I^* , I^* \right] , ( \varphi , s ) \in \mathbb{T}^2 \right\},
\end{equation*}
where $( p_0, q_0)$ are the separatrices to the saddle equilibrium point of the pendulum
\begin{equation*}
\left(p_0(t) , q_0(t)\right) = \left(\frac{\pm 2}{\cosh t} , 4 \arctan\textrm{e}^{ \pm t}\right).
\end{equation*}

In the perturbed case, i.e., for small $\varepsilon > 0 $, $W^u(\tilde{\Lambda}_{\varepsilon})$
and $W^s(\tilde{\Lambda}_{\varepsilon})$ do not coincide (this is the so-called
splitting of separatrices), and every local transversal intersection between them gives
rise to a (local) scattering map which is simply the correspondence between a
past asymptotic motion in the NHIM to the corresponding future asymptotic motion following a homoclinic orbit.
Since the NHIM has also an inner dynamics, an adequate combination of these two dynamics on the NHIM, the inner one and the outer one provided by the scattering map, generates the global instability (also called in short \emph{Arnold diffusion}) as long as the outer dynamics does not preserve the invariant objects of the inner dynamics. 

The inner motion is described in Section~\ref{sec:inner}, the scattering maps in Section~\ref{sec:scattering} and the absence of
invariant sets in both dynamics is checked in Section~\ref{sec:diffusion},
which also includes the proof of Theorem~\ref{theo:main_theo}.
Section~\ref{sec:piecewise} deals with the construction of a piecewise smooth global scattering map
which is introduced as a possible new tool to design fast and simple paths of global instability.
We finish this Introduction with some remarks
about the necessity of the assumptions, as well as other features of the scattering map and
a discussion of the model chosen and related work.

\subsection{Necessity of the assumptions}
If the determinant $\Delta:=k_1 l_2-k_2 l_1$ or some coefficient $a_1$, $a_2$ vanishes, for instance, if there is only one harmonic in $g$, there is no
global instability for the action $I$. Indeed, looking at the equations associated to Hamiltonian~\eqref{eq:hamil_system}
\begin{align}
\label{eq:hamil_equations}
\dot{q} &= \pm p&\dot{p} &= \left[ \pm 1 + \varepsilon\left(a_1\cos(k_1\varphi + l_1 s) + a_2\cos(k_2\varphi + l_2 s)\right)\right]\sin q\nonumber \\
\dot{\varphi} &= I&\dot{I} &= \varepsilon\cos q \left(k_1 a_1\sin(k_1\varphi + l_1 s) + k_2 a_2\sin(k_2\varphi + l_2 s)\right)\\
\dot{s} &= 1\nonumber&&
\end{align}
this is clear for $k_1=k_2=0$, since in this case $I$ is a constant of motion. 
If $k_1$ or $k_2\neq 0$, say $k_1\neq 0$,
the change of variables
\begin{equation*}
\bar{\varphi} = k_1\varphi + l_1 s, \quad\quad r\bar{\varphi} - \bar{s} = k_2\varphi + l_2s, \quad\quad \bar{I}= k_1 I + l_1,
\end{equation*}
where $r = k_2/k_1$ can be assumed to satisfy $0\leq r \leq 1$ without loss of generality, casts system~\eqref{eq:hamil_equations} into
\begin{align*}
\dot{q} &=\pm p& \dot{p} &= \left[\pm 1  + \varepsilon \left( a_1 \cos\bar{\varphi} + a_2\cos(r\bar{\varphi} - \bar{s})\right)\right] \sin q\\
\dot{\bar{\varphi}} &= \bar{I}&\dot{\bar{I}} &= \varepsilon k_1^2\cos q\left(a_1\sin \bar{\varphi}  + r a_2 \sin(r\bar{\varphi} - \bar{s})\right)\\
\dot{\bar{s}} &= \Delta/k_1&&
\end{align*}
which is a Hamiltonian system with the Hamiltonian given by
\begin{equation}
\label{eq:bar_hamil_system}
\bar{H}_{\varepsilon}(p,q,\bar{I},\bar{\varphi},\bar{s}) = \pm \left(\frac{p^2}{2} + \cos q - 1 \right) + \frac{\bar{I}^2}{2} + \varepsilon k_1^2\cos q\left(a_1\cos\bar{\varphi} + a_2\cos(r\bar{\varphi} - \bar{s})\right).
\end{equation}
If $\Delta=0$ Hamiltonian~\eqref{eq:bar_hamil_system} is autonomous with 2 degrees of freedom,
and therefore a global drift for the action $I$ is not possible. Only drifts of size $\sqrt{\varepsilon}$ are possible due to KAM theorem.
Analogously one easily checks that for $a_1a_2=0$ Hamiltonian~\eqref{eq:hamil_system} is integrable or autonomous.

\subsection{Reduction of the harmonic types}
Under the hypothesis $\left(k_1 l_2-k_2 l_1\right) a_1 a_2\neq 0$ of Theorem~\ref{theo:main_theo}, we first
notice that the case $k_2 = 0 $ of Theorem~\ref{theo:main_theo} is already proven in~\cite{Delshams2017}.
Indeed, $k_2 = 0 $ implies $r:=k_2/k_1 = 0$ and it turns out from~\eqref{eq:bar_hamil_system} that Hamiltonian~\eqref{eq:hamil_system}
is equivalent to the one with $k_1=1, k_2 = 0,l_1=0,l_2= 1$:
\begin{equation}
H_{\varepsilon}(p , q , I , \varphi , t) = \pm\left( \frac{p^2}{2} + \cos q  -1 \right) + \frac{I^2}{2}
+ \varepsilon \cos q \left(a_1 \cos \varphi + a_2\cos s\right),
\label{eq:old_hamil_system}
\end{equation}
which is just the Hamiltonian studied in \cite{Delshams2017}.
Therefore, we only need to prove
Theorem~\ref{theo:main_theo} for $k_1 k_2 \neq 0$ or equivalently for $r\in (0,1]$.
For the sake of clarity we will explain in full detail and prove Theorem~\ref{theo:main_theo}
along Sections~\ref{sec:inner}-\ref{sec:diffusion} just for $r=1$, which by~\eqref{eq:bar_hamil_system}
is equivalent to the case $k_1=1, k_2 = 1,l_1=0,l_2 = -1$:
\begin{equation}
H_{\varepsilon}(p , q , I , \varphi , t) = \pm\left( \frac{p^2}{2} + \cos q  -1 \right) + \frac{I^2}{2}
+ \varepsilon \cos q \left(a_1 \cos \varphi + a_2\cos (\varphi - s)\right).
\label{eq:new_hamil_system}
\end{equation}
To finish the proof of Theorem~\ref{theo:main_theo}, in Section~\ref{sec:diffusion} we will sketch the
modifications needed for the case $r\in (0,1)$.

\subsection{Scattering map types}
By the definition given at the beginning of Section~\ref{sec:scattering}, a scattering map is
in principle only \emph{locally} defined, that is, for a small ball of values of the variables
$(I,\varphi,s)$ or $(I,\theta=\varphi-Is)$, since it depends on a non-degenerate critical point
$\tau^*=\tau^*(I,\varphi,s)$ of a real function~\eqref{eq: SM_critical_point}, depending smoothly
on the variables $(I,\varphi,s)$, already introduced in~\cite{Seara2006}.
In the study carried out in Section~\ref{sec:scattering}, it
will be described whether, in terms of the parameter $\mu:=a_1/a_2$ and the variable $I$, a local
scattering map can or cannot be smoothly defined for all the values of the angles $(\varphi,s)$ or
$\theta=\varphi-Is$, becoming thus a \emph{global} or \emph{extended} scattering map. This description will depend essentially
on a geometrical characterization of the function~$\tau^*(I,\varphi,s)$ in terms of the
intersection of \emph{crests} and \emph{NHIM lines}, following~\cite{Delshams2011}.
Any degeneration of the critical point $\tau^*=\tau^*(I,\varphi,s)$ may give rise to more
non-degenerate critical points and a bifurcation to \emph{multiple} local scattering maps
or to a non global scattering map. Different critical points $\tau^*=\tau^*(I,\varphi,s)$
give rise to different local scattering maps, and putting together different local scattering maps,
one can sometimes obtain \emph{piecewise smoth} global scattering maps, which are very useful
to design paths of instability for the action $I$, and are simply called diffusion paths.

For instance, in the paper~\cite{Delshams2017} devoted to the Hamiltonian~\eqref{eq:old_hamil_system},
it was proven that for $0<\mu=a_1/a_2<0.625$, there exist two different global scattering maps.
Among the different kinds of associated orbits of these scattering maps, there appeared
two of them called \emph{highways}, where the drift of the action $I$ was very fast and simple.
As will be described in Section~\ref{sec:scattering}, such highways do not appear for
Hamiltonian~\eqref{eq:new_hamil_system}. Nevertheless, as will be proven in Section~\ref{sec:piecewise},
there exist piecewise smooth global scattering maps, and the possible diffusion along
the discontinuity sets opens the possibility of applying the theory of
piecewise smooth dynamical systems~\cite{Filippov88}.

\subsection{About the model chosen and related work}
Hamiltonian~\eqref{eq:hamil_system} is a standard example of an \emph{a priori unstable} Hamiltonian
system~\cite{ChierchiaGallavotti} formed by a pendulum, a rotor and a perturbation. It is usual
in the literature to choose a perturbation depending periodically only on the
positions---which turn out to be angles in our case---and time. Our perturbation $h(q,\varphi,s)$ \eqref{eq:g_general_case}
is a little bit special since it is a product of a function $f(q)$ times a function $g(\varphi,s)$. This
choice makes easier the computations of the Poincar\'e-Melnikov potential~\eqref{eq:our_meln_potential},
which is based on the Cauchy's residue theorem.
Theorem~\ref{theo:main_theo} can be easily generalized to any trigonometric polynomial or meromorphic function $f(q)$,
although the computations of poles of high order become more complicated. In the same way, it could also be
generalized to more general perturbations $h(q,\varphi,s)$, as long that $h$ is a trigonometric polynomial or meromorphic in $q$.
The dependence on more than two harmonics gives rise to the appearance of more resonances in the inner dynamics,
which requires more control of their sizes, see for instance~\cite{DelshamsS97,Delshams2009}.
Apart from more difficulty in the computations of the Poincar\'e-Melnikov
potential and the inner Hamiltonian, we do not foresee substantial changes, so we believe that
Hamiltonian~\eqref{eq:hamil_system} could be considered as a paradigmatic example
of an \emph{a priori unstable} Hamiltonian system.

This paper is a natural culmination of~\cite{Delshams2017}, which dealt
with the simpler Hamiltonian~\eqref{eq:old_hamil_system}, and where a detailed description of NHIM lines
and crests was carried out. An ``optimal'' estimate of the diffusion time close to some special orbits of the scattering map,
called \emph{highways}, was also given there. The study in this paper of Hamiltonian~\eqref{eq:new_hamil_system} is more
complicated, due to a greater complexity of the evolution of the NHIM lines and crests with respect to the action $I$
and the parameters of the system. In particular, the absence of highways prevents us of showing an estimate
of diffusion time close to them.

The paper~\cite{Delshams2017} also contains a fairly extensive list of references about global instability. Let us
simply mention some new references that are not there, like~\cite{Davletshin2016} which contains a similar approach
to the function~$\tau^*$ of~\cite{Seara2006} and the crests of~\cite{Delshams2011}, and the recent preprints
\cite{GelfreichT14,LazzariniMS15,Marco16,GideaM17,Cheng17}
involving the geometrical method or variational methods.

We finish this introduction by noticing that in this paper we stress the interaction between {\NH} lines and crests,
since this allows us to describe the diverse scattering maps, as well as their domains, that appear in our problem.
In more complicated models of Celestial Mechanics the Melnikov potential is not available.
In these cases the computations of scattering maps rely on the numerical computation of invariant manifolds of a NHIM
or some of its selected invariant objects, and the search of diffusion orbits is performed in a more crafted way (see \cite{Canalias2006,delshams2008b,delshams2013,capinski2017,delshams201629}).
%=======================================================================================
% INNER DYNAMICS
%======================================================================================
\section{Inner dynamics}
\label{sec:inner}

The inner dynamics is derived from the restriction of $H_{\varepsilon}$ in \eqref{eq:new_hamil_system} and its equations to $\tilde{\Lambda}$, that is,
\begin{equation}
K(I,\varphi,s) = \frac{I^{2}}{2} + \varepsilon\left(a_{1}\cos\varphi + a_{2}\cos(\varphi -s)\right), \label{eq:hamiltonian_inner}
\end{equation}
and differential equations
\begin{equation}
\dot{\varphi} = I \quad\quad \dot{s} = 1\quad\quad\dot{I} =\varepsilon\left(a_{1}\sin\varphi + a_{2}\sin(\varphi-s)\right).\label{eq:equations_inner}
\end{equation}

Note that in this case the inner dynamics is slightly more complicated to describe than in \cite{Delshams2017} where there was just one resonance, namely, in $I=0$.
In the current case we have two resonant regions of size $\mathcal{O}(\sqrt{\varepsilon})$ where secondary KAM tori appear.
To describe these regions, we use normal forms as in \cite{Seara2006}.

Consider the autonomous extended Hamiltonian
\begin{equation}
\overline{K}(I,A,\varphi,s) = \frac{I^2}{2} + A + \varepsilon\left(a_{1}\cos\varphi + a_{2}\cos(\varphi -s)\right),\label{eq:bar_K}
\end{equation}
with the differential equations
\begin{align*}
\dot{\varphi} =& I &\dot{I} =&\varepsilon\left(a_{1}\sin\varphi + a_{2}\sin(\varphi-s)\right)\\
\dot{s} =& 1  &\dot{A} =& -\varepsilon a_{2}\sin(\varphi-s).
\end{align*}
This system is equivalent to the system represented by \eqref{eq:hamiltonian_inner}+\eqref{eq:equations_inner}.
We wish to eliminate the dependence on the angle variables.
Consider a change of variables $\varepsilon$-close to the identity $$(\varphi,s , I , A ) = g(\phi , \sigma , J , B) = (\phi , \sigma , J , B) + \mathcal{O}(\varepsilon)$$
such that it is the one-time flow for a Hamiltonian $\varepsilon G$, i.e.,  $g = g_{t=1}$, where $g_t$ is solution of
\begin{equation*}
\frac{dg_t}{dt} = J_0\nabla\varepsilon G\circ g_t, \text{ where } J_0 \text{ is the symplectic matrix }\begin{pmatrix}
 0&1 \\
 -1&0
\end{pmatrix}.
\end{equation*}
Composing $\overline{K}$ with $g$ and expanding in a Taylor series around $t=0$, one obtains
\begin{equation*}
\overline{K}\circ g = \overline{K} + \left\{\overline{K},\varepsilon G\right\} + \frac{1}{2}\left\{\left\{\overline{K},\varepsilon G\right\} , \varepsilon G\right\} + \dots,
\end{equation*}
where $\left\{\cdot\right\}$ is the Poisson bracket.
Using the expansion \eqref{eq:bar_K} of $\overline{K}$, the equation above can be written as
\begin{equation}
\begin{split}
\overline{K}\circ g = \frac{J^2}{2} + B + \varepsilon\left( a_{1}\cos\phi + a_{2}\cos(\phi-\sigma) + \left\{\frac{J^2}{2} + B ,G\right\} \right) \\+ \frac{\varepsilon^2}{2}\left\{\left\{\frac{J^2}{2} + B,G\right\},G\right\} + \mathcal{O}(\varepsilon^3).\label{eq:exp_K}
\end{split}
\end{equation}
We want to find $G$ such that $a_{1}\cos\phi + a_{2}\cos(\phi-\sigma) + \left\{\frac{J^2}{2} + B ,G\right\}=0$, or equivalently,
\begin{equation*}
J\frac{\partial G}{\partial\phi} + \frac{\partial G}{\partial \sigma } = a_{1}\cos\phi + a_{2}\cos(\phi-\sigma).
\end{equation*}
Given $a <b<1$, consider any function $\Psi \in C^{\infty}(\mathbb{R})$ satisfying $\Psi(x) = 1$ for $x\in \left[-a,a\right]$ and $\Psi(x) = 0$ for $\left|x\right| > b$ and introduce
\begin{equation*}
G(J,B,\phi,\sigma) := \frac{a_{1}}{J}\left(1-\Psi(J)\right)\sin\phi + \frac{a_{2}}{J-1}\left(1-\Psi(J-1)\right)\sin(\phi - \sigma),
\end{equation*}
Substituting the above function $G(J,B,\phi,\sigma)$ in \eqref{eq:exp_K} we have
\begin{eqnarray}
\overline{K}\circ g =\frac{J^2}{2} + B + \mathcal{O}(\varepsilon^2) , \label{eq:invariant_tori_nr}
\end{eqnarray}
for $J,J-1 \notin [-b,b]$.
For $J\in [-a,a]$,
\begin{equation}
\overline{K}\circ g = \frac{J^2}{2} + B + \varepsilon a_{1}\cos\phi + \mathcal{O}(\varepsilon^2). \label{eq:ressonance_1}
\end{equation}
Finally, for $J-1 \in [-a,a]$,
\begin{equation}
\overline{K}\circ g = \frac{J^2}{2} + B + \varepsilon a_{2}\cos(\phi-\sigma) + \mathcal{O}(\varepsilon^2). \label{eq:ressonance_2}
\end{equation}
From \eqref{eq:ressonance_1} and \eqref{eq:ressonance_2}, one sees that on $J = 0$ and $J = 1$ there are resonances of first order in $\varepsilon$ with a pendulum-like behavior.

Coming back to the original variables, three kinds of invariant tori are obtained.
For the first order resonance $ I = 0 $, there is a positive $\overline{a}$ such that the invariant tori are given by $F^{0}(I,\varphi,s)=~\text{constant}$ with
\begin{equation}
F^0(I,\varphi,s) = \frac{I^2}{2} + \varepsilon a_{1}\cos\varphi  + \mathcal{O}(\varepsilon^2).\label{eq:invariant_tori_i=0}
\end{equation}
for $I \in \left[-\overline{a},\overline{a}\right]$

Analogously, for the first order resonance $ I = 1 $, with
\begin{equation*}
F^1(I,\varphi,s) = \frac{(I-1)^2}{2} + \varepsilon a_{2} \cos(\varphi-s) + \mathcal{O}(\varepsilon^2),\label{eq:invariant_tori_i=1}
\end{equation*}
for $I-1 \in \left[-\overline{a},\overline{a}\right]$.

\begin{remark}
As commented in \cite{Seara2006}, there exists a \emph{secondary} resonance in $I = 1/2$, but the size of the gap in its resonant region is much smaller than the size of gaps in resonant regions associated to $I = 0$ and $I= 1$.
\end{remark}

\begin{remark}In a more general case with $r\neq 1$, the resonances take place in $I = 0$ and $I = 1/r$.\label{rem:r_1}
\end{remark}
From \eqref{eq:invariant_tori_nr}, on the non-resonant region the invariant tori has equations $F^{\text{nr}}(I) =~\text{constant}$ with
\begin{equation*}
F^{\text{nr}}(I) = \frac{I^2}{2} + \mathcal{O}(\varepsilon^2).\label{eq:invariant_tori_nr_b}
\end{equation*}
An illustration of the inner dynamics is displayed in Figure \ref{fig:inner_dynamics}.

%================ figure:  Inner
\begin{figure}[h]
\centering
\includegraphics[scale=0.33]{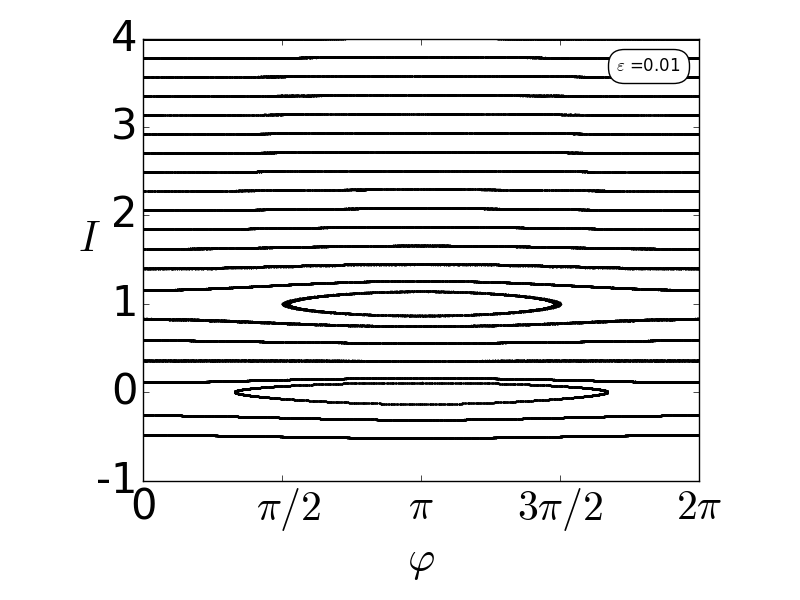}
\caption{Plane $\varphi \times I$ of inner dynamics for $\mu = 0.75$ and $\varepsilon = 0.01$.}\label{fig:inner_dynamics}
\end{figure}
%====================

%===============================================================
% SCATTERING MAP
%===============================================================
\section{Scattering map}
\label{sec:scattering}

\subsection{Definition of scattering map }

We are going to explore the properties of the scattering maps of Hamiltonian \eqref{eq:new_hamil_system}.
The notion of scattering map on a {\NH} was introduced in \cite{Delshams2000}.
Let $W$ be an open set of $\left[-I^* , I^*\right] \times \mathbb{T}^2$ such that the invariant manifolds of the {\NH} $\tilde{\Lambda}$ introduced in~\eqref{eq:NHIM_manifold} intersect transversally along a homoclinic manifold $\Gamma = \left\{\tilde{z}(I,\varphi,s;\varepsilon) , (I,\varphi,s)\in W\right\}$ so that for any $\tilde{z}\in\Gamma$ there exist unique $\tilde{x}_{+,-} = \tilde{x}_{+,-}(I,\varphi,s;\varepsilon)\in\tilde{\Lambda}$ such that $\tilde{z}\in W_{\varepsilon}^{s}(x_-)\cap W_{\varepsilon}^{u}(\tilde{x}_+)$.
Let
$$H_{+,-} = \bigcup\left\{\tilde{x}_{+,-}(I,\varphi , s ; \varepsilon) : (I,\varphi,s)\in W\right\}.$$
The scattering map associated to $\Gamma$ is the map
\begin{eqnarray*}
S: H_- & \longrightarrow & H_+\\
\tilde{x}_- &\longmapsto & S(\tilde{x}_-) = \tilde{x}_+.
\end{eqnarray*}

For the characterization of the scattering maps, it is required to select the homoclinic manifold $\Gamma$ and this is done using the Poincar\'{e}-Melnikov theory. From \cite{Seara2006,Delshams2011}, we have the following proposition

\begin{proposition}\label{prop:melnpot}
Given $(I,\varphi,s)\,\in\,\left[-I^{*},I^{*}\right]\,\times\,\mathbb{T}^{2}$, assume that the real function
\begin{equation}\label{eq: SM_critical_point}
\tau\,\in\,\mathbb{R}\,\longmapsto\,\mathcal{L}(I,\varphi-I\,\tau,s-\tau)\,\in\,\mathbb{R}
\end{equation}
has a non degenerate critical point $\tau^{*}\, =\, \tau^*(I,\varphi,s)$, where
\begin{equation*}
\mathcal{L}(I,\varphi,s):=\int_{-\infty}^{+\infty}\left(f(q_{0}(\sigma)) - f(0)\right)g(\varphi+I\sigma,s+\sigma;0)d\sigma.
\end{equation*}
Then, for $0\,<\,\varepsilon$ small enough, there exists a unique transversal homoclinic point $\tilde{z}$ to $\tilde{\Lambda}_{\varepsilon}$ of Hamiltonian~\eqref{eq:hamil_system}, which is $\varepsilon$-close to the point
$\tilde{z}^{*}(I,\varphi,s)\,=\,(p_{0}(\tau^{*}),q_{0}(\tau^{*}),I,\varphi,s)\,\in\,W^{0}(\tilde{\Lambda})$:
\begin{equation*}
\tilde{z}=\tilde{z}(I,\varphi,s)=(p_{0}(\tau^{*})+O(\varepsilon), q_{0}(\tau^{*})+O(\varepsilon),I,\varphi,s)\,\in\,W^{u}(\tilde{\Lambda_ {\varepsilon}})\,\pitchfork\,W^{s}(\tilde{\Lambda_{\varepsilon}}).
\end{equation*}
\end{proposition}

The function $\mathcal{L}$ is called the \emph{Melnikov potential} of Hamiltonian \eqref{eq:hamil_system}.
For the concrete Hamiltonian~\eqref{eq:new_hamil_system} it takes the form
\begin{equation}
\mathcal{L}(I,\varphi,s) = A_{1}(I)\cos\varphi + A_{2}(I)\cos(\varphi - s),\label{eq:our_meln_potential}
\end{equation}
where
\begin{equation*}
A_{1}(I) = \frac{2\pi I a_{1}}{\sinh(\pi I/2)} \quad \text{ and }\quad A_{2}(I) = \frac{2\pi(I-1)a_{2}}{\sinh(\pi(I-1)/2)}.
\end{equation*}
The homoclinic manifold $\Gamma$ is characterized by the function $\tau^*(I,\varphi,s)$.
Once a $\tau^*(I,\varphi,s)$ is chosen, which under the conditions of Proposition~\ref{prop:melnpot}, is locally smoothly well defined, by the geometric properties of the scattering map, see \cite{Delshams2008,Delshams2009,Delshams2011}, the scattering map has the explicit local form

\begin{equation*}
S(I,\varphi,s) = \left(I + \varepsilon\frac{\partial L^*}{\partial \varphi}(I,\varphi,s) + \mathcal{O}(\varepsilon^2) , \varphi - \varepsilon\frac{\partial L^*}{\partial I}(I,\varphi,s) +\mathcal{O}(\varepsilon^2) , s \right),
\end{equation*}
where
\begin{equation}
L^*(I,\varphi,s) = \mathcal{L}(I,\varphi - I\tau^*(I,\varphi,s) , s-\tau^*(I,\varphi,s)).\label{eq:L^*-def}
\end{equation}

Notice that the variable $s$ is fixed under the scattering map.
As a consequence \cite{Delshams2011,Delshams2017}, introducing the variable
\begin{equation*}
\theta = \varphi - I s 
\end{equation*}
and defining the \emph{reduced Poincar\'{e} function}
\begin{equation}\label{eq:reduced_poincare_function}
\mathcal{L}^{*}(I,\theta) := L^*(I,\varphi - Is , 0) = L^*(I,\varphi,s),
\end{equation}
in the variables $(I,\theta)$, the scattering map has the simple form
\begin{equation*}
\mathcal{S}(I,\theta) = \left( I + \varepsilon\frac{\partial \mathcal{L}^*}{\partial\theta}(I,\theta) + \mathcal{O}(\varepsilon^2) , \theta - \varepsilon\frac{\partial \mathcal{L}^*}{\partial I}(I,\theta) + \mathcal{O}(\varepsilon^2)   \right),
\end{equation*}
so up to $\mathcal{O}(\varepsilon^2)$ terms, $\mathcal{S}(I,\theta)$ is the $\varepsilon$ times flow of the \emph{autonomous} Hamiltonian $-\mathcal{L}^*(I,\theta)$.
In particular, the iterates under the scattering map follow the level curves of $\mathcal{L}^*$ up to $\mathcal{O}(\varepsilon^2)$.

\subsection{Crests and {\NH} lines}

We have seen that the function $\tau^*$ plays a central role in our study.
Therefore, we are interested in finding the critical points $\tau^* = \tau^*(I,\varphi,s) $ of function \eqref{eq: SM_critical_point}.
For our concrete case \eqref{eq:our_meln_potential}, $\tau^*$ is a solution of
\begin{equation}
 I A_{1}(I)\sin(\varphi - I\tau^*) + (I-1)A_{2}(I)\sin(\varphi - s - (I-1)\tau^*)=0.\label{eq:tau^*}
\end{equation}
This equation can be viewed from two equivalently geometrical viewpoints.
The first one is that to find $\tau^* = \tau^*(I,\varphi,s)$ satisfying \eqref{eq:tau^*} for any $(I,\varphi,s)\in \left[-I^* , I^* \right]\times\mathbb{T}^2$ is the same as to look for the extrema of $\mathcal{L}$ on  the \emph{{\NH} line}
\begin{equation}
R(I,\varphi,s) = \left\{(I,\varphi - I\tau , s-\tau) : \tau \in \mathbb{R}\right\}.\label{eq:nhim_line}
\end{equation}
\begin{remark}
Since $(\varphi,s)\in\mathbb{T}^2$, $R(I,\varphi,s)$ is a closed line if $I \in \mathbb{Q}$ and it is a dense line on
$\{I\}\times\mathbb{T}^2$ if $I \notin \mathbb{Q}$.
\end{remark}
The other viewpoint is that, fixing $(I,\varphi,s)$, a solution $\tau^*$ of \eqref{eq:tau^*} is equivalent to finding intersections between a {\NH} line \eqref{eq:nhim_line} and a curve defined by
\begin{equation}
IA_1( I ) \sin\varphi + (I - 1)A_2(I)\sin(\varphi - s) = 0.\label{eq:def_crests_2}
\end{equation}
These curves are called \emph{crests}, and in a general way can be defined as follows.

\begin{definition}\label{def:crest}
\cite{Delshams2011} We define by \emph{Crests} $\mathcal{C}(I)$ the curves on $(I,\varphi,s)$, $(\varphi,s)\in \mathbb{T}^2$, such that
\begin{equation}
\frac{\partial \mathcal{L}}{\partial \tau}(I , \varphi - I \tau ,s - \tau)|_{\tau = 0} = 0,\label{eq:def_crests}
\end{equation}
or equivalently,
$$I\frac{\partial \mathcal{L}}{\partial\varphi}(I,\varphi,s) + \frac{\partial\mathcal{L}}{\partial s}(I,\varphi,s)= 0.$$
\end{definition}

As in our case $\mathcal{L}(I , \varphi - I\tau , s - \tau ) = A_1(I)\cos(\varphi - I\tau) + A_2( I ) \cos( \varphi  - s - (I - 1) \tau )$, equation \eqref{eq:def_crests} takes the form
\eqref{eq:def_crests_2}.
Introducing
\begin{equation}
\sigma = \varphi - s, \label{eq:def_sigma}
\end{equation}
equation \eqref{eq:def_crests_2} can be rewritten as
\begin{equation}
 \mu \alpha( I ) \sin\varphi + \sin\sigma =0,\label{eq:crests_equation}
\end{equation}
for $I \neq 1$, where
\begin{equation}
\mu = \frac{a_{1}}{a_{2}}\quad\text{and}\quad \alpha(I) =\frac{I^2\sinh(\frac{\pi}{2}(I-1))}{(I-1)^2\sinh(\frac{\pi I}{2})}.\label{eq:mu_alpha}
\end{equation}

From now on, when we refer to crests $\mathcal{C}(I)$ we mean the set of points $(I,\varphi , \sigma)$ satisfying equation~\eqref{eq:crests_equation}.
See an illustration in Fig.~\ref{fig:qual_plot}.

\begin{remark}
In \cite{Delshams2017} the crests were described on the plane $(\varphi , s )$, whereas now such curves lie on the plane $(\varphi , \sigma)$.
Besides, differently from the cases studied in \cite{Delshams2011,Delshams2017}, the function $\alpha(I)$ is not defined for all $I$.
More precisely, it is not defined for $I = 1$. For this value of $I$, equation \eqref{eq:crests_equation} is not adequate,
and one has to use \eqref{eq:def_crests_2} to check that for $I = 1$ the crests are just two vertical straight lines on the plane $(\varphi,\sigma)$ given by $\varphi = 0$ and $\varphi = \pi$.
\end{remark}
\begin{remark}
For Hamiltonian \eqref{eq:bar_hamil_system} and $r\in (0,1)$, $\alpha_{r}(I)$ is not defined for $I = 1/r$ and is given by
\begin{equation*}
\alpha_r(I) = \frac{I^2\sinh\left(\frac{\pi}{2}(rI - 1)\right)}{(rI-1)^2\sinh\left(\frac{\pi I}{2}\right)}.
\end{equation*}
\label{rem:r_2}
\end{remark}
We are interested in understanding the behavior of these crests because, as we have seen in previous works \cite{Delshams2011,Delshams2017}, their intersection with the {\NH} lines determine the existence and behavior of scattering maps.

From \eqref{eq:crests_equation}, when $\left|\alpha(I)\right| <1/\left|\mu\right|$, $\sigma$ can be written as a function of $\varphi$ for all $\varphi\in\mathbb{T}$ on the crest $\mathcal{C}(I)$.
On the other hand, if $\left|\alpha(I)\right| > 1/\left|\mu\right|$, $\varphi$ can be written as a function of $\sigma$ for all $\sigma\in\mathbb{T}.$
These two conditions give us two kinds of crests: \emph{horizontal} for $\left|\alpha(I)\right| <1/\left|\mu\right|$ and \emph{vertical} for  $\left|\alpha(I)\right| > 1/\left|\mu\right|$.
These names are due to their forms on the plane $(\varphi , \sigma)$.
We consider the same characterization used in \cite{Delshams2017}:
\begin{itemize}
\item For $\left|\alpha(I)\right|\, <\, 1/\left|\mu\right|$, there are two horizontal crests $\sigma = \xi_{\M,\m}(I,\varphi)$
$$\mathcal{C}_{\M,\m}(I)=\{(I,\varphi,\xi_{\M,\m}(I,\varphi)):\varphi\in\mathbb{T}\},$$
\begin{eqnarray}\label{eq:cristas_06}
\xi_{\M}(I,\varphi)&=&-\arcsin(\mu\alpha(I)\sin \varphi )  \quad\quad\mod{2\pi}\\
\xi_{\m}(I,\varphi)&=&\arcsin(\mu\alpha(I)\sin \varphi )+\pi \quad\quad\mod{2\pi}.\nonumber
\end{eqnarray}
\item  For $\left|\alpha(I)\right|\, >\, 1/\left|\mu\right|$, there are two vertical crests  $\varphi = \eta_{\M,\m}(I,\sigma)$
$$\mathcal{C}_{M,m}(I)=\{(I,\eta_{M,m}(I,\sigma),\sigma):\sigma\in\mathbb{T}\},$$
\begin{eqnarray*}
\eta_{M}(I,\sigma)&=&-\arcsin(\sin \sigma/\left(\mu \alpha(I)\right))\quad\quad\mod{2\pi} \label{eq:eta_definition}\\
\eta_{m}(I,\sigma)&=&\arcsin(\sin \sigma/\left(\mu \alpha(I)\right))+\pi\quad\quad\mod{2\pi}.\nonumber
\end{eqnarray*}
\end{itemize}

\begin{remark}
\label{rmk:singular}
$\left|\alpha(I)\right| = 1/\left|\mu\right|$ is a singular or bifurcation case.
In this case, the crests are straight lines and are not differentiable in $\varphi = \pi/2$ and $\varphi = 3\pi/2$.
See Fig.~6 of~\cite{Delshams2017}.
\end{remark}
\begin{remark}
The crest containing the point $(\varphi , \sigma) = (0,0)$ will be denoted by $\mathcal{C}_{\M}(I)$  and the crest containing the point $(\varphi , \sigma) = (\pi , \pi)$ by $\mathcal{C}_{\m}(I)$.
\end{remark}

Note that the function $\left|\alpha(I)\right|$ is not bounded, indeed
\begin{equation*}
\lim_{I\rightarrow 1}\left|\alpha(I)\right| = +\infty.
\end{equation*}
This implies that for any $\mu$ there exists a neighborhood $U$ of $I = 1$ such that for all $I\in U$ the crests are vertical.
On the other hand, since $\alpha( 0 ) = 0 $ there exists a neighborhood $V$ of $I = 0$ such that for all $I\in V$ the crests are horizontal.
We notice here a remarkable difference with the Hamiltonians studied in \cite{Delshams2011,Delshams2017}, where, for $\left|\mu\right| \leq 0.97$, all the crests are horizontal for all $I$.

Now take a look at the properties of the function $\alpha(I)$ introduced in \eqref{eq:mu_alpha} to describe under which conditions in $\mu$ the crests are horizontal or vertical.
First of all, observe that for $I\neq 1$, $\alpha(I)$ is smooth and $\alpha'(I)\neq 0, $ and for $I = 1$ $\alpha(I)$ is not bounded, indeed it has a vertical asymptote
\begin{equation*}
\lim_{I\rightarrow 1^-} \alpha(I) = -\infty \quad\quad \text{ and }\quad\quad\lim_{I\rightarrow 1^+}\alpha(I) = +\infty.
\end{equation*}

Given a $\mu\neq0$, since $\alpha(0) = 0$, there exists a unique $I_c \in (0,1)$ such that $\left|\alpha(I)\right| = 1/\left|\mu\right|$.
So, the crests are horizontal for $I\in\left[ 0 , I_c\right)$ and vertical for $I\in(I_c , 1)$.

Others important limits are
\begin{equation*}
\lim_{I\rightarrow -\infty}\alpha(I) = \exp(\pi/2) \quad\quad \text{ and }\quad\quad \lim_{I\rightarrow + \infty } \alpha(I) = \exp(-\pi/2).
\end{equation*}
The first limit implies that $\left|\alpha(I)\right| < \exp( \pi/2)$ for $I\in\left(-\infty , 0\right)$.
Thus, if $\exp(\pi/2)\leq 1/\left|\mu\right|$ the crests are horizontal for $I\in(-\infty , 0)$.
Otherwise, if $1/\left|\mu\right| < \exp(\pi/2 )$, there exists a unique $I_{\text{l}}\in(-\infty , 0)$ such that $\left|\alpha(I)\right| = 1/\left|\mu\right|$ and the crests are vertical for $I\in(-\infty , I_{\text{l}})$ and horizontal for $I\in\left(I_{\text{l}} ,0\right)$.

The second limit implies that $\left|\alpha(I)\right| > \exp(-\pi/2)$ for $I\in(1 , +\infty)$.
Then, if $\exp( -\pi/2 )\geq 1/\left|\mu\right|$, the crests are vertical for $I\in\left[1 , +\infty\right)$.
if $\exp( -\pi/2 )< 1/\left|\mu\right|$, there exists a unique $I_{\text{r}}\in(1 , + \infty)$, such that the crests are vertical for any $I$ in $\left[1 , I_{\text{r}}\right)$ and horizontal for $I\in(I_\text{r} , +\infty)$.

Summarizing, for $1/\left|\mu\right| \geq \exp(\pi/2)$, crests are horizontal for $I\in\left(-\infty , I_{\text{c}}\right)\cup(I_{\text{r}} , +\infty)$ and vertical for $I\in\left(I_{\text{c}},I_{\text{r}}\right)$.
For $\exp(-\pi/2) < 1/\left|\mu\right| < \exp( \pi/2)$, crests are horizontal for $I\in(I_{\text{l}}, I_{\text{c}})\cup(I_{\text{r}} , +\infty)$ and vertical for $I\in(-\infty , I_{\text{l}})\cup(I_{\text{c}} , I_{\text{r}})$.
Finally, if $ 1/\left|\mu\right| < \exp( -\pi/2)$, crests are horizontal for $I\in(I_{\text{l}} , I_{\text{c}})$ and vertical for $I\in(-\infty , I_{\text{l}})\cup(I_{\text{c}, + \infty})$.

\begin{remark}
For $r\in (0,1)$, $\alpha_r(I)$ is not bounded on a neighbourhood of the resonance $I = 1/r$, i.e., $\lim_{I \rightarrow 1/r^{-}} \alpha_r(I) = -\infty $ and $\lim_{I\rightarrow 1/r^{+}}\alpha_r(I) = +\infty$.
The same behavior takes place for $r = 1$ and close to $I = 1$.
On the other hand, for $I\rightarrow \pm \infty$, $\alpha_r(I)$ has the same behavior as in the case for $r = 0$, $\lim_{I \rightarrow \pm\infty}\alpha_r(I) = 0$.
This implies that for any value of $\mu$, for $I$ close enough to $I = 1/r$ the crests are vertical, and for $\left|I\right| $ large enough the crests are horizontal. \label{rem:r_3}
\end{remark}

\paragraph{Example} To illustrate this discussion, we present a concrete example.
Taking $\mu = 0.5$,  we have $\exp(-\pi/2)< 1/\mu = 2<\exp(\pi/2)$.
In this case we have $I_{\text{l}}\approx -1.807$, $I_{\text{c}} \approx 0.701$ and $I_{\text{r}} \approx 1.367$.
The crests are horizontal in $(-1.807 , 0.701)\cup(1,367 , + \infty)$ and vertical in $(-\infty , -1.807)\cup( 0.701 , 1.367 )$.
We emphasize that this scenario is very different from the case in \cite{Delshams2017}.
There, for $\mu = 0.5$ the crests are horizontal for all $I$.
\\

Now, we are going to focus on the transversality of the intersection between {\NH} lines $R(I,\varphi , s)$ and crests $\mathcal{C}(I)$.
On the plane $(\varphi , \sigma)$ the {\NH} lines can be written as
\begin{equation}
R_{I}(\varphi, \sigma) = \{( \varphi - I\tau , \sigma -(I-1)\tau) , \tau\in\mathbb{R}\},\label{eq:nhim_line_sigma}
\end{equation}
so that its slope is $(I-1)/I$ in such plane.
Therefore, there exists an intersection between {\NH} lines and crests that is not transversal if, and only if, there exists a tangent vector of $\mathcal{C}(I)$ at a point that is parallel to $(I,I-1)$, or, using the parameterizations,
\begin{equation*}
\frac{\partial\xi}{\partial \varphi}(I,\varphi) = \frac{I-1}{I} \quad\quad\text{ or } \quad\quad \frac{\partial \eta}{\partial\sigma}(I,\sigma) = \frac{I}{I-1}.
\end{equation*}

Considering a \emph{horizontal} parameterization of $\mathcal{C}(I)$, the tangency condition is equivalent to
\begin{equation*}
\frac{\pm\alpha(I)\mu\cos\varphi}{\sqrt{1-\mu^2\alpha^2(I)\sin^2\varphi}} = \frac{I-1}{I}.
\end{equation*}
Therefore, there exists a $\varphi$ satisfying the above condition if, and only if,
\begin{equation*}
\left|\beta(I)\right| \geq \frac{1}{\left|\mu\right|},\quad\text{ where }\quad\beta(I) = \frac{I\alpha(I)}{I-1}\label{eq:condition_tan_hor}
\end{equation*}
and $\varphi$ takes the form
\begin{equation*}
\varphi =\pm \arctan\left(\sqrt{\frac{\beta(I)^2-(1/\mu)^2}{(1/\mu)^2-\alpha(I)^2}}\right).
\end{equation*}
In an analogous way, for a \emph{vertical} parameterization $\eta(I,\sigma)$, there are tangencies if, and only if,
\begin{equation*}
\left|\beta(I)\right|\leq \frac{1}{\left|\mu\right|} \quad\quad \text{ with }\quad\quad \sigma = \pm \arctan\left(\left|\frac{I-1}{I}\right|\sqrt{\frac{(1/\mu)^2 - \beta(I)^2}{\alpha(I)^2-(1/\mu)^2}}\right). \label{eq:condition_tan_ver}
\end{equation*}

\begin{remark}
Observe that in both cases, horizontal and vertical crests, there are tangencies if, and only if,
\begin{equation*}
\left(\left|\alpha(I)\right| - \frac{1}{\left|\mu\right|}\right)\left( \left|\beta(I)\right| - \frac{1}{\left|\mu\right|} \right) < 0.
\end{equation*}
\end{remark}

The function $\left|\beta(I)\right|$ is smooth in $\mathbb{R}\setminus\left\{1\right\}$ and $d\left|\beta(I)\right|/dI = 0$ only for $I=0$.
Besides, we have
\begin{equation*}
\lim_{I\rightarrow 1}\left|\beta(I)\right| = +\infty, \quad\quad \lim_{I\rightarrow - \infty}\left|\beta(I)\right| = \exp(\pi/2)\quad\text{ and }\quad \lim_{I\rightarrow +\infty}\left|\beta(I)\right| = \exp( -\pi/2).
\end{equation*}
Therefore, there are three possibilities:
\begin{itemize}
\item for $1/\left|\mu\right| \geq \exp(\pi/2)$, there exist $I_0\in(1/2 , 1)$ and $I_+\in(1,+\infty)$ such that $I_0$ and $I_+$ are solutions of $\left|\beta(I)\right| - 1/\left|\mu\right|=0$.
Besides, $\left|\beta(I)\right|<1/\left|\mu\right|$ for $I\in(-\infty , I_0)\cup(I_+,+\infty)$ and $\left|\beta(I)\right|> 1/\left|\mu\right|$ for $I\in(I_0 , 1)\cup(1 , I_+)$.
\item for $\exp(-\pi/2 )< 1/\left|\mu\right| < \exp(\pi/2)$, there exist $I_-\in(-\infty , 0)$, $I_0\in(0 , 1)$ and $I_+\in(1,+\infty)$ such that $I_-$, $I_0$ and $I_+$ are solutions of $\left|\beta(I)\right| - 1/\left|\mu\right|=0$.
Besides, $\left|\beta(I)\right|<1/\left|\mu\right|$ for $I\in(I_- , I_0)\cup(I_+,+\infty)$ and $\left|\beta(I)\right|> 1/\left|\mu\right|$ for $I\in(-\infty,I_-)\cup(I_0 , 1)\cup(1 , I_+)$.
\item For $ 1/\left|\mu\right| \leq \exp(-\pi/2)$, there exist $I_-\in(-\infty , 0)$ and $I_0\in(0 , 1/2)$ such that $I_-$ and $I_0$ are solutions of $\left|\beta(I)\right| - 1/\left|\mu\right|=0$.
Besides, $\left|\beta(I)\right|<1/\left|\mu\right|$ for $I\in(I_- , I_0)$ and $\left|\beta(I)\right|> 1/\left|\mu\right|$ for $I\in(-\infty,I_-)\cup(I_0 , 1)\cup(1 , \infty)$.
\end{itemize}

Putting together this description of $\left|\beta(I)\right|$ with the study about vertical and horizontal crests and adding that
\begin{eqnarray*}
\left|\beta(I)\right| < \left|\alpha(I)\right| &  \forall I\in(-\infty , 0)\cup(0,1/2);\\
\left|\beta(I)\right| > \left|\alpha(I)\right| &  \forall I\in(1/2 , 1)\cup(1,+\infty);\\
\left|\beta(0)\right| = \left|\alpha(0)\right| = 0&  \left|\beta(1/2)\right| = \left|\alpha(1/2)\right| = 1
\end{eqnarray*}
we can state the proposition below.

\begin{proposition}\label{prop:cristas_behave}
Consider the two crests $\mathcal{C}(I)$ defined by \eqref{eq:crests_equation} and the {\NH} line $R_{I}(\varphi,\sigma)$ defined in \eqref{eq:nhim_line} for Hamiltonian~\eqref{eq:new_hamil_system}.
\begin{itemize}
\item For $\left|\mu\right| \leq \exp(-\pi/2)$, there exist $I_{\text{b}} < I_{\text{a}}< I_{\text{A}} <I_{\text{B}}$ such that
\begin{itemize}
\item for $I<I_{\text{b}}$ or  $I_{\text{B}}<I$, $\mathcal{C}(I)$ are horizontal and intersect transversally any $R_{I}(\varphi , \sigma)$;
\item for $I_{\text{b}} \leq I < I_{\text{a}}$ or $I_{\text{A}} < I\leq I_{\text{B}}$, the crests $\mathcal{C}(I)$ are horizontal, but now, there exist tangencies between $\mathcal{C}(I)$ and two {\NH} lines $R_{I}(\varphi , \sigma)$;
\item for $I_{\text{a}} < I < I_{\text{A}}$, the crests $\mathcal{C}(I)$ are vertical and intersect transversally any $R_{I}(\varphi , \sigma)$.
\end{itemize}
\item For $\exp(-\pi/2) < \left| \mu \right| < \exp( \pi/2)$ there exist $I_{\text{b}} < I_{\text{a}} < I _{\text{c}} \leq I_{\text{C}}< I_{\text{A}} <I_{\text{B}}$ such that
\begin{itemize}
\item for $I < I_{\text{b}}$ or $I_{\text{C}} < I < I_{\text{A}}$, $\mathcal{C}(I)$ are vertical and intersect transversally any $R_{I}(\varphi , \sigma)$;
\item for $I_{\text{b}}\leq I < I_{\text{a}}$, the crests $\mathcal{C}(I)$ are vertical and there exist tangencies between $\mathcal{C}(I)$ and two {\NH} lines $R_{I}(\varphi , \sigma)$;
\item for $I_{\text{a}} < I < I_{\text{c}}$ or $I_{\text{B}} < I$, $\mathcal{C}(I)$ are horizontal and  intersect transversally any $R_{I}(\varphi , \sigma)$;
\item for $I_{\text{A}} \leq I \leq I_{\text{B}}$, the crests $\mathcal{C}(I)$ are horizontal and there exist tangencies between $\mathcal{C}(I)$ and two {\NH} lines $R_{I}(\varphi , \sigma)$;
\item  for $I_{\text{c}} \leq I \leq I_{\text{C}}$,  if $I_{\text{c}} <1/2$, the crests $\mathcal{C}(I)$ are vertical and there
exist tangencies between $\mathcal{C}(I)$ and $R_{I}(\varphi,\sigma)$.
If $I_{\text{c}} = 1/2$, from the properties of $\alpha(I)$ and $\beta(I)$ this interval is just one point.
If $I_{\text{c}}> 1/2$, the crests $\mathcal{C}(I)$ are horizontal and there exist tangencies.
\end{itemize}
\item For $\left|\mu\right| \geq \exp( \pi/2 )$ there exist $I_{\text{b}} < I_{\text{a}} < I_{\text{A}} < I_{\text{B}}$ such that
\begin{itemize}
\item for $I < I_{\text{b}}$ or $I_{\text{B}} < I$, $\mathcal{C}(I)$ are vertical and intersect transversally any $R_{I}(\varphi , \sigma)$;
\item for $ I_{\text{b}} \leq I < I_{\text{a}}$ or $ I_{\text{A}} < I \leq I_{\text{B}}$, the crests $\mathcal{C}(I)$ are vertical and there exist tangencies between $\mathcal{C}(I)$ and two {\NH} lines $R_{I}(\varphi , \sigma)$;
\item for $I_{\text{a}} < I < I_{\text{A}}$, the crests $\mathcal{C}(I)$ are horizontal and intersect transversally any $R_{I}(\varphi , \sigma)$.
\end{itemize}
\end{itemize}
\end{proposition}
\begin{remark}
Note that we are not considering the singular case $\left|\alpha(I)\right| = 1/\left|\mu\right|$ described in Remark~\ref{rmk:singular}.
\end{remark}

\paragraph{Example} Again, to illustrate this proposition, we take the case with $\mu = 0.5$, see Fig.~\ref{fig:alpha_and_beta}.
In this case, we have $\left|\beta(I)\right| = 1/\mu$ for $I \approx -2.942,\, 0.595,\, 1.85$
and
\begin{itemize}
\item for $I \in (-\infty , -2.942)\cup \left(0.701 , 1\right)\cup(1 , 1.367)\Rightarrow \left\{\begin{matrix}
 \left|\alpha(I)\right| > 1/\left|\mu\right|& \Rightarrow  & \text{vertical crests} \\
\left|\beta(I)\right| > 1/\left|\mu\right| &  \Rightarrow & \text{no tangencies}
\end{matrix}\right. $

\item for $I \in \left[ -2.942 , -1.807 \right) \Rightarrow \left\{\begin{matrix}
 \left|\alpha(I)\right| > 1/\left|\mu\right|& \Rightarrow  & \text{vertical crests} \\
\left|\beta(I)\right| \leq 1/\left|\mu\right| & \Rightarrow & \text{tangencies}
\end{matrix}\right.$

\item for $I \in (-1.807 , 0.595 )\cup( 1.85 ,+\infty)\Rightarrow \left\{\begin{matrix}
 \left|\alpha(I)\right| < 1/\left|\mu\right|& \Rightarrow  & \text{horizontal crests} \\
\left|\beta(I)\right| < 1/\left|\mu\right| & \Rightarrow & \text{no tangencies}
\end{matrix}\right.$

\item for $I \in \left[0.595 , 0.701 \right)\cup \left( 1.367 , 1.85\right]\Rightarrow \left\{\begin{matrix}
 \left|\alpha(I)\right| < 1/\left|\mu\right|& \Rightarrow  & \text{horizontal crests} \\
\left|\beta(I)\right| \geq 1/\left|\mu\right| & \Rightarrow & \text{tangencies}
\end{matrix}\right.$
\end{itemize}

Once more, we compare with the Hamiltonian~\eqref{eq:old_hamil_system} studied in \cite{Delshams2017}.
For Hamiltonian~\eqref{eq:old_hamil_system} and $\mu = 0.5$ there is no tangency, but for Hamiltonian~\eqref{eq:new_hamil_system}
we can find tangencies for horizontal and vertical crests.
Indeed, for Hamiltonian~\eqref{eq:old_hamil_system} and any $0<\left|\mu\right| < 0.625$ there is no tangency,
whereas for any $\mu\neq 0$ there are tangencies for Hamiltonian~\eqref{eq:new_hamil_system}.

%================================ figure: alpha_and_betha
\begin{figure}[h]
\centering
\subfigure[$\left|\alpha(I)\right|$ and $\left|\beta(I)\right|$:$\mu = 0.5$, $I_b \approx -2.942$, $I_a \approx -1.807$, $I_c \approx 0.595$, $I_C \approx 0.701$, $I_A \approx 1.367$ and $I_B \approx 1.85$]{\includegraphics[scale=0.33]{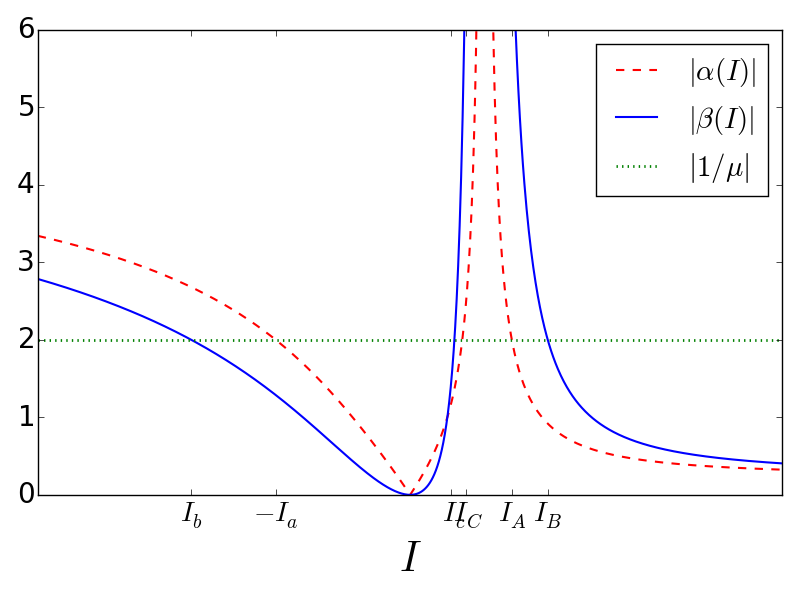}\label{fig:alpha_and_beta}}
\qquad
\subfigure[$\left|\alpha_r(I)\right|$ and $\left|\beta_r(I)\right|$: $\mu = 0.5$ and $r = 0.5$.]{\includegraphics[scale=0.33]{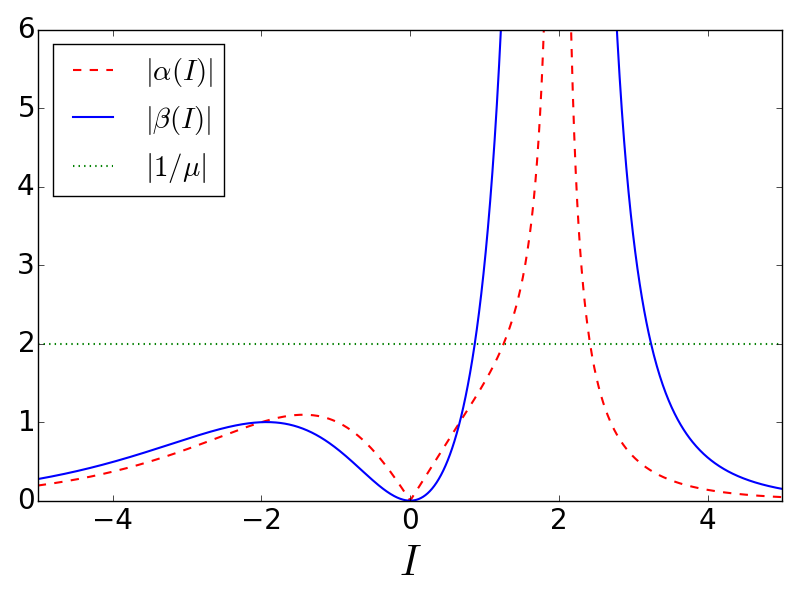}\label{fig:alpha_and_beta_r_0.5}}
\qquad
\caption{$\left|\alpha(I)\right|$ and $\left|\beta(I)\right|$ : Behavior of the crests and tangencies.}
\end{figure}
%===========================================================

\begin{remark}
For $r\in (0 ,1 ),$ $\beta_r(I)$ is defined by $\beta_r(I) = I\alpha_r(I)/(rI - 1)$.
In this case, $\lim_{I\rightarrow 1/r}\left|\beta_r(I)\right| = +\infty$ and $\lim_{I \rightarrow \pm \infty}\left|\beta_r(I)\right| = 0$.
In Fig.~\ref{fig:alpha_and_beta_r_0.5}, a comparison between the functions $\alpha_r(I)$, $\beta_r(I)$ and the straight line $1/\left|\mu\right|$ for $r=1/2$ is displayed.\label{rem:r_4}
\end{remark}

For each crest, where it is well defined, there exists, at least, a  value $\tau^*$ such that
\begin{equation*}
(\varphi - I\tau^* , \sigma - (I - 1)\tau^*) = (\varphi - I\tau^* , \xi(I , \varphi - I\tau^*))\text{ or } (\eta(I ,\sigma - (I-1)\tau^* ) , \sigma - (I-1)\tau^*),
\end{equation*}
which means that $R_{I}(\varphi , \sigma)\cap \mathcal{C}(I)\neq \emptyset$.
This intersection is intrinsically associated to a homoclinic orbit to the NHIM.
To make a choice about how to take such $\tau^*$ is to choose in which homoclinic manifold $\Gamma$
the homoclinic points $\tilde{z}^*$ lie.
Even more, it is to choose what scattering map we are going to use.

\subsection{Construction of scattering maps}

We have now several goals. First, to explain, given $(I,\theta)$, how to find the intersection between one of
the {\NH} lines and one of the two crests, and consequently, to define the function $\tau^*$.
Second, to show how each crest can give rise to many scattering maps.
And third, to explain the different scattering maps or combinations of them that can be defined.

Let us first study the intersection between {\NH} lines and crests.
From the definition of the function $\tau^* = \tau^*(I,\varphi,s)$ in equation \eqref{eq:tau^*} and the definition of a {\NH} line $R(I,\varphi,s)$ in \eqref{eq:nhim_line} and a crest $\mathcal{C}(I)$ in Definition~\ref{def:crest}, it turns out that
\begin{equation*}
R(I,\varphi,s)\cap\mathcal{C}(I) = \left\{\left(I,\varphi- I\tau^*(I,\varphi,s),s - \tau^*(I,\varphi,s) \right)\right\}.
\end{equation*}
Moreover, from the equation satisfied by the function $\tau^*$, one can get (see Eq.~(3.12) in \cite{Delshams2017}) that for any $\gamma$
\begin{equation*}
\tau^*(I,\varphi - I\gamma, s -\gamma) = \tau^*(I,\varphi,s) - \gamma.
\end{equation*}
In particular, for the change \eqref{eq:def_sigma} $s = \varphi - \sigma$ and $\gamma = \varphi - \sigma$ one gets
\begin{equation}\label{eq:tau_sigma}
\tau^*(I,\varphi,\varphi-\sigma) = \tau^*(I,\theta) + \varphi - \sigma,
\end{equation}
where $\theta = \varphi - Is = (1-I)\varphi + I\sigma$.
In the variables $(I,\varphi,\sigma)$, taking into account the expression \eqref{eq:nhim_line_sigma} for the {\NH} lines $R(I,\varphi,\varphi - \sigma)$ and again equation \eqref{eq:tau^*} satisfied by the $\tau^*(I,\varphi,s)$, we have that
\begin{align*}
R(I,\varphi, \varphi-\sigma)\cap\mathcal{C}(I) &= \left\{\left(I,\varphi - I\tau^*(I,\varphi,\varphi-\sigma) , \sigma - (I-1)\tau^*(I,\varphi,\varphi -\sigma )\right)\right\}\\
&=\left\{\left(I,\theta - I\tau^*(I,\theta),\theta - (I-1)\tau^*(I,\theta)\right)\right\},
\end{align*}
where \eqref{eq:tau_sigma} has been used, and $\theta = (1-I)\varphi + I\sigma$.

From a geometrical point of view, to find an intersection between a {\NH} line and a crest, one throws from a point $(\theta , \theta)$ on the plane $(\varphi,\sigma)$ a straight line with slope $(I-1)/I$, until it touches the crest $\mathcal{C}(I)$. The function $\tau^*(I,\theta)$ is the time spent to go from a point $(\theta,\theta)$ in the diagonal $\sigma = \varphi$ up to $\mathcal{C}(I)$ with a velocity vector $\mathbf{v} = -(I , I-1)$, see Fig.~\ref{fig:qual_plot}.

%================ figure:  qualitative_plot
\begin{figure}[h]
\centering
\includegraphics[scale=0.4]{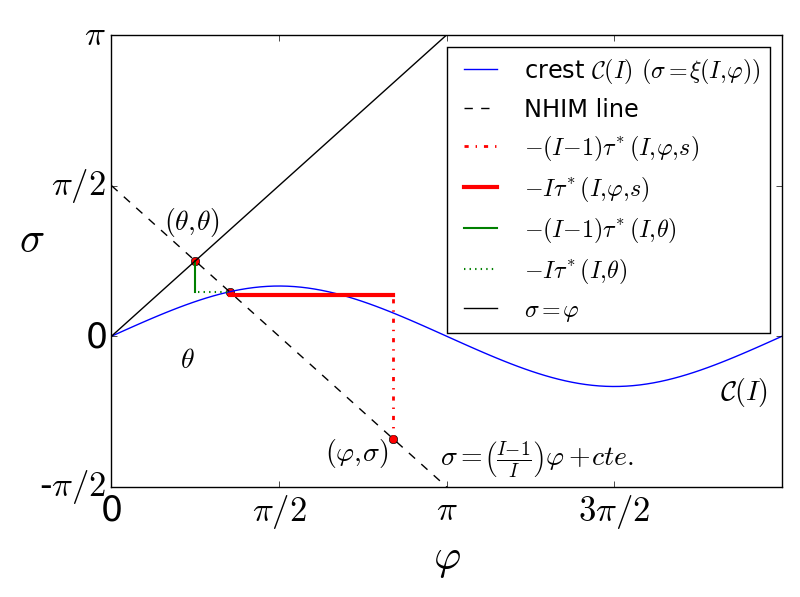}
\caption{Finding $\tau^*(I,\theta)$ using the straight line $\sigma = \varphi$.\label{fig:qual_plot}}
\end{figure}
%====================
One has to decide the direction for $\tau^*$ using the idea explained above.
For example, if we are on a point on the straight line $\sigma = \varphi$, we have to decide if we go up or go down along the {\NH} line, i.e., to look for a negative or a positive $\tau^*(I,\theta)$ (to look at the past or the future).
In both cases we are going to detect an intersection with the desired crest, but, in general, different choices give rise to different scattering maps, because we are looking for different homoclinic invariant manifolds $\Gamma$.

To show another difference between scattering maps from the choice of $\tau^*$ we begin by introducing each kind of scattering map.
The first one is inspired in \cite{Delshams2011} and \cite{Delshams2017} for $\left|\mu\right| < 0.97$.
In these cited cases all scattering maps studied were associated to one of the horizontal crests like in \eqref{eq:cristas_06}.
In the same way, we can separate completely the scattering maps associated to the horizontal crests and the scattering maps associated to the vertical crests.
Notice that the scattering maps associated to horizontal crests are defined only for values of $I$ satisfying $\left|\alpha(I)\right|< 1/\left|\mu\right|$ whereas the scattering maps associated to the vertical crests are defined only for values of $I$ satisfying $\left|\alpha(I)\right| > 1/\left|\mu\right|$.

As noted previously, crests are vertical in a neighborhood of $I = 1$ for any value of $\mu$.
Therefore, there is no scattering map associated to a horizontal crest close to $I=1$.
Analogously, since $\left|\alpha(0)\right| = 0$, crests are horizontal in a neighborhood of $I = 0$ for any value of $\mu$ and, therefore, there is no scattering map associated to a vertical crest close to $I = 0$.
This implies that these ``horizontal'' or ``vertical'' scattering maps are just locally defined, in other words, they are not defined on the whole plane $(\theta , I)$.
This motivates to define \emph{global scattering maps}.
Global scattering maps are important because they describe the outer dynamics for large intervals of $I$ and are defined as follows
\begin{definition}
A scattering map $\mathcal{S}(I,\theta)$ is called  a \emph{global scattering map} if it is defined on all $\theta\in\mathbb{T}$ for any fixed $I$.
\end{definition}

Note that $\mathcal{S}(I,\theta)$ is a global scattering map as long as $\tau^*(I , \theta)$ is a global function, i.e., defined on all $\theta\in\mathbb{T}$ for any fixed $I$.
If $\tau^*(I,\theta)$ is smoothly defined, the same will happen to $\mathcal{S}(I,\theta)$.
Tangencies between \NH\, lines and crests, as well as discontinuities in their intersections give rise to non-smooth scattering maps.

\begin{remark}
For instance, in the paper~\cite{Delshams2017} devoted to the Hamiltonian~\eqref{eq:old_hamil_system},
it was proven that for $0<\mu=a_1/a_2<0.625$, there exist two different global scattering maps.
Let us add that for $0.625\leq \mu < 0.97$, due to the existence of tangencies between the {\NH} lines and the crests, there appear two or six scattering maps.
Such \emph{multiple} scattering maps are indeed piecewise smooth global scattering maps, see Figs.~9--11 of~\cite{Delshams2017}.
Their discontinuities lie along the \emph{tangency locus} and were avoided there to construct diffusion paths,
just for the sake of simplicity.
\end{remark}

For Hamiltonian \eqref{eq:new_hamil_system}, to extend scattering maps which are in principle only locally defined we have now two options: to combine a scattering map associated to a horizontal crest with a scattering map associated to a vertical crest or to extend the previously called ``horizontal'' or ``vertical'' scattering maps.
Although the first option may provide a global scattering map, they may appear complex discontinuity sets which give rise to a complicated phase space.

The second option is to apply the same idea used in \cite{Delshams2017} when we defined the scattering map ``with holes''.
When $\left|\alpha(I)\right| >1/\left|\mu\right|$, the horizontal crests are no longer defined for all $\varphi\in\mathbb{T}$, indeed, they become vertical crests defined for all $\sigma\in\mathbb{T}$. Nevertheless, the vertical crests are formed by pieces of horizontal crests.
This implies that even for these values of $I$ we can use $\xi$ given in \eqref{eq:cristas_06} to parameterize some intersections between $R(I,\varphi ,\sigma)$ and $\mathcal{C}(I)$.
As we can see in Fig.~\ref{fig:hor_vert_prox}, the vertical and horizontal crest $\mathcal{C}_{\M}$ are very close in a neighbourhood of $\varphi = 0$.
When we have a bifurcation from horizontal to vertical crests (or vice versa), it is natural just to change the parameterization from $\xi_{\M}$ to $\eta_{\M}$ for these values of $\varphi$.
With this choice the orbits of the scattering map are continuous for $\theta$ close to $0$ or $2\pi$.
The same happens with $\xi_{\m}$ and $\eta_{\m}$ for values of $\varphi$ close to $\pi$.
Scattering maps associated to horizontal crests for values of $I$ satisfying $\left|\mu\alpha(I)\right| < 1$ are defined for all $\varphi\in\mathbb{T}$.
The extension of them to values of $I$  for $\varphi\in\mathbb{T}$ such that $\left|\mu\alpha(I)\sin\varphi\right| < 1$ are called \emph{extended scattering maps}.
\begin{definition}
A scattering map $\mathcal{S}(I,\theta)$ is called an \emph{extended scattering map} if it is associated to horizontal crests for which $\left|\mu\alpha(I)\right|<1$, and is continuously extended to the pieces of the vertical crests where they behave as horizontal crests, that is, for the values $\varphi$ such that $\left|\mu\alpha(I)\sin\varphi\right| < 1$.
\end{definition}
Since we have already seen in Proposition~\ref{prop:cristas_behave} that there exist tangencies between {\NH} lines and crests for any value of $\mu$, there are no global scattering maps for Hamiltonian \eqref{eq:new_hamil_system}. However, there exist
extended scattering maps with a domain large enough to provide diffusion paths.
%================================ figure: pedaços de cristas
\begin{figure}[h]
\centering
\subfigure[A piece of $\xi_{\M}(I,\varphi)$ for $I = 0.68$.]{\includegraphics[scale=0.27]{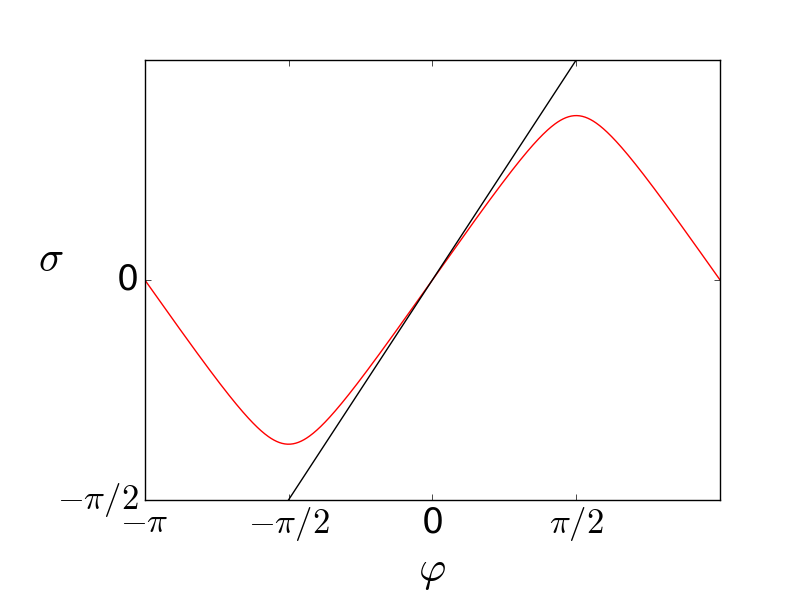}\label{fig:hor_crest_part}}
\qquad
\subfigure[A piece of $\eta_{\M}(I,\sigma)$ for $I = 0.72$.]{\includegraphics[scale=0.27]{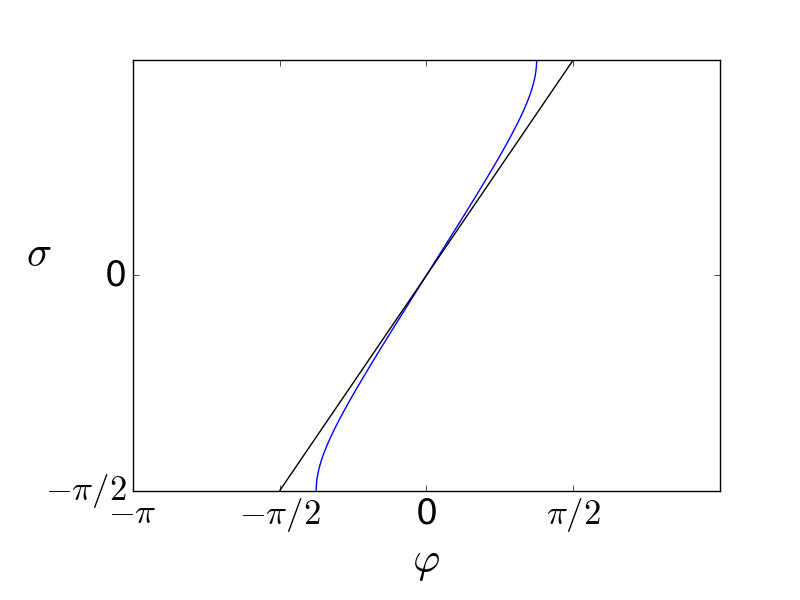}\label{fig:vert_crest_part}}
\qquad
\caption{Comparison between $\xi_{\M}(I,\varphi)$ and $\eta_{\M}(I,\sigma)$ for $\mu = 0.5$, $I = 0.68$ and $I = 0.72$ respectively.\label{fig:hor_vert_prox}}
\end{figure}
%===========================================================

To illustrate the current scenario we will display the level curves of the reduced Poincar\'{e} function $\mathcal{L}^*(I,\theta)$ defined in \eqref{eq:reduced_poincare_function}, which up to $\mathcal{O}(\varepsilon^2)$ contain orbits of the scattering map $\mathcal{S}(I,\theta)$.
We begin by considering $\mu = 0.6$ and the horizontal crest $\mathcal{C}_{\M}(I)$. In Fig.~\ref{fig:go_down} we display the scattering map built using $\tau^*$ defined by the first intersection between $R_I(\varphi,\sigma)$ and $\mathcal{C}_{\M}(I)$ from $\sigma =\varphi$ going down along $R_I(\varphi,\sigma)$.
In Fig.\ref{fig:go_up}, we use a similar idea, but now, form $\sigma =\varphi$ going up along $R_I(\varphi,\sigma)$.
Alternatively, if we choose $\tau^*$ with minimal absolute value, independently of going up or down, we obtain the scattering map plotted on Fig.~\ref{fig:min_tau}.
In this last case, there are orbits of the scattering maps that are not smooth in $\theta = \pi$.
This happens because we change the homoclinic manifold $\Gamma$, so we are using, indeed, two different scattering maps.
In \cite{Delshams2017} we chose scattering maps associated to a function $\tau^*$ with the minimal absolute value,
which were called \emph{primary} scattering maps.
This example show us that is not enough to say what crest is associated to a scattering map, but it is also necessary to make
explicit the criterion used for $\tau^*$ (going up or down along the {\NH} lines, or choosing a minimal $\left|\tau^*\right|)$.

%================================ figure: scattering_maps
\begin{figure}[h]
\centering
\subfigure[Going down along the {\NH} lines $R_I(\varphi,\sigma)$.]{\includegraphics[scale=0.21]{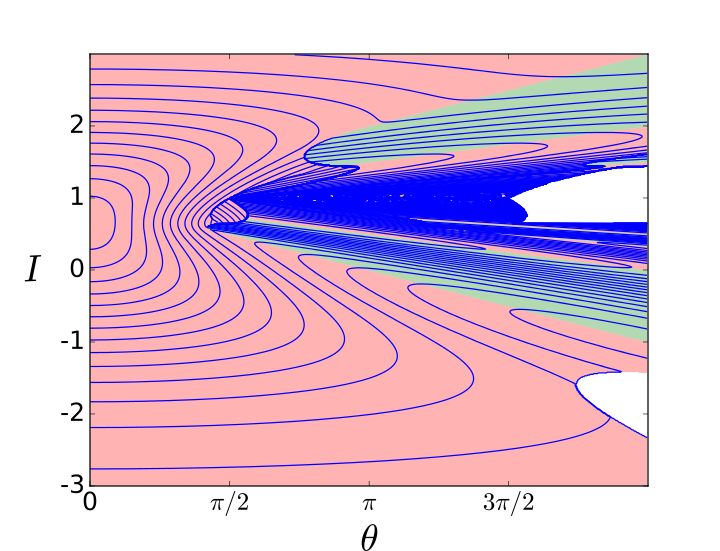}\label{fig:go_down}}
\qquad
\subfigure[Going up along the {\NH} lines $R_I(\varphi,\sigma)$.]{\includegraphics[scale=0.21]{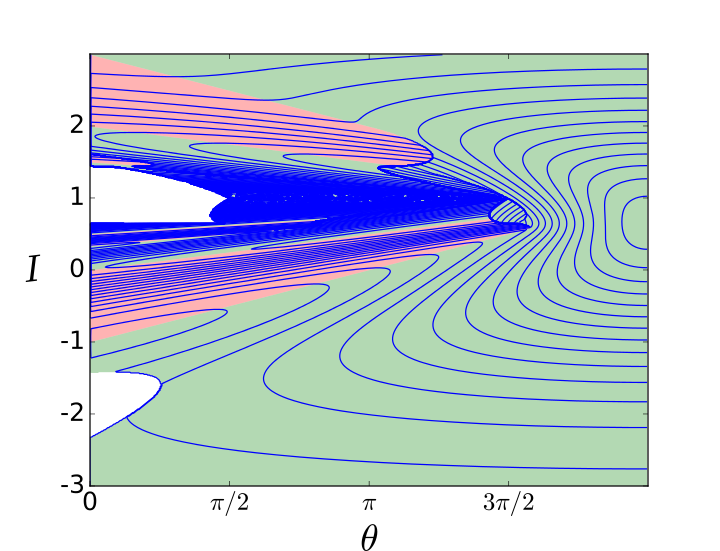}\label{fig:go_up}}
\qquad
\subfigure[Minimal absolute value of $\tau^*$.]{\includegraphics[scale=0.21]{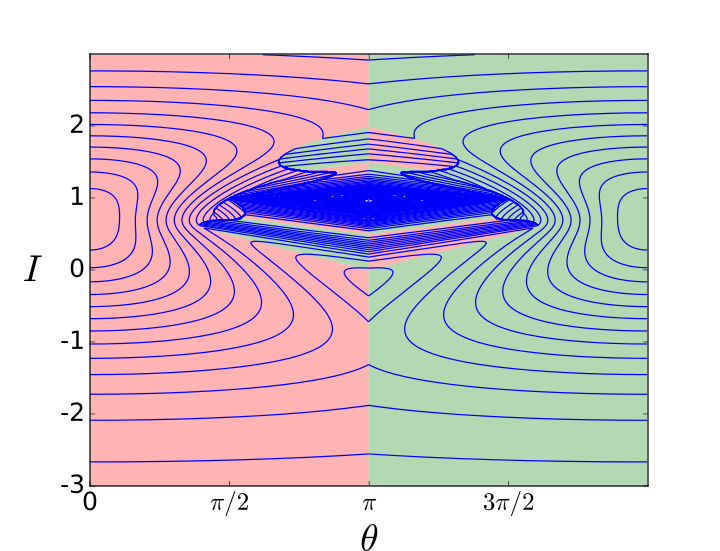}\label{fig:min_tau}}
\qquad
\caption{Different phase space of scattering maps $\mathcal{S}(I,\theta)$ associated to the same horizontal crest $C_{\M}(I)$, for $\mu = 0.6$ and $\varepsilon = 0.01$. The orbits of scattering maps are represented by the blue lines which are, up to $\mathcal{O}(\varepsilon^2)$, level sets of the reduced Poincar\'{e} function $\mathcal{L}^*(I,\theta)$.
In the red zones the values of $I$ on such orbits decrease, in the green one the values of $I$ increase.
The white regions are regions where $\left|\mu\alpha(I)\sin\varphi\right|>1$ is satisfied.}
\end{figure}

The next lemma is a good example about the criteria for $\tau^*(I,\theta)$ and its consequences, and is used to prove Proposition \ref{prop:orb_cres}.
Before, a new notation is introduced.
An \emph{even} subindex $k$ is assigned to the branches $\mathcal{C}_k(I)$ of $\mathcal{C}_{\M}(I)$ when considering $\sigma,\,\varphi \in\mathbb{R}$
\begin{equation*}
\xi_{k}(I,\varphi) = -\arcsin\left(\alpha(I)\mu\sin\varphi\right) + k\pi\quad\text{ and  }\quad\eta_{k} = -\arcsin\left(\frac{\sin\sigma}{\alpha(I)\mu}\right) + k\pi
\end{equation*}
and an \emph{odd} subindex $k$ to the branches $\mathcal{C}_k(I)$ of $\mathcal{C}_{\m}(I)$ when considering $\sigma,\,\varphi \in\mathbb{R}$
\begin{equation*}
\xi_{k}(I,\varphi) = \arcsin\left(\alpha(I)\mu\sin\varphi\right) + k\pi\quad\text{ and  }\quad\eta_{k} = \arcsin\left(\frac{\sin\sigma}{\alpha(I)\mu}\right) + k\pi.
\end{equation*}

We notice that the crests $\mathcal{C}(I)$ are naturally defined for $(\varphi,\sigma)\in\mathbb{T}^2$ and give rise to two different crests $\mathcal{C}_{\M}(I)$, $\mathcal{C}_{\m}(I)$ (except for the singular case $\left|\mu\alpha(I)\right| = 1$).
When we run now over real values of $\varphi,\,\sigma$, we may have an \emph{infinite} number of crests $\mathcal{C}_k(I)$, where even (odd) values of $k$ are assigned to the branches of $\mathcal{C}_{\M}(I)$ ($\mathcal{C}_{\m}(I)$). Among them,
we are going to use $\mathcal{C}_{0}(I)$, $\mathcal{C}_{1}(I)$ and $\mathcal{C}_{2}(I)$.

\begin{lemma}Let $\mathcal{L}_{\zeroM}^*$ and $\mathcal{L}_{\oneM}^*$ be reduced Poincar\'{e} functions associated to the same crest $\mathcal{C}(I)$, where for $\mathcal{L}_{\zeroM}^*$ we look at the first intersection points ``under'' $\sigma = \varphi$, that is, with $\mathcal{C}_0(I)$, and for $\mathcal{L}_{\oneM}^*$ we look at the first intersection points ``over'' $\sigma = \varphi$, that is, with $\mathcal{C}_2(I)$.
Then we have
\begin{equation}
\frac{\partial\mathcal{L}_{\zeroM}^*}{\partial\theta}\left(I,\theta\right) = -\frac{\partial\mathcal{L}_{\oneM}^*}{\partial\theta}\left(I,2\pi - \theta\right).\label{eq:lemma}
\end{equation}
\end{lemma}

\begin{remark}
We say ``under"  $\sigma = \varphi$  and ``over" $\sigma = \varphi$ for intersection points going down or up along $R_I(\varphi,\sigma)$, respectively on $(\varphi , \xi_{0}(I,\varphi))$ and $(\varphi,\xi_{2}(I,\varphi))$, because when the horizontal crest $\mathcal{C}_{\M}(I)$ is defined for all $\varphi\in\mathbb{T}$ the graphs $(\varphi , \xi_{0}(I,\varphi))$ of $\mathcal{C}_0(I)$ and  $(\varphi,\xi_{2}(I,\varphi))$ of $\mathcal{C}_2(I)$ are under and over the straight line $\sigma = \varphi$.
\end{remark}

\begin{proof}
Let $\mathcal{L}^*$ be a reduced Poincar\'{e} function~\eqref{eq:reduced_poincare_function}-\eqref{eq:our_meln_potential}, then
\begin{equation*}
\frac{\partial\mathcal{L}^*}{\partial\theta}\left(I,\theta\right) = \frac{A_1(I)\sin(\theta - I\tau^*(I,\theta))}{I-1}.\label{eq:diff_L__d_theta}
\end{equation*}
So, equation \eqref{eq:lemma} is satisfied if, and only if
\begin{equation}
\sin(\theta - I\tau_{\zeroM}^*(I,\theta)) = \sin(\theta -I(-\tau_{\oneM}^*(I,2\pi-\theta))).\label{eq:equi_lemma}
\end{equation}

We assume that the crest is horizontal and given by the graph of $\xi_{\M}$, the other cases are analogous.
Indeed, we are going to use
\begin{equation}
\xi_{\zeroM}(I,\varphi) = -\arcsin(\mu\alpha(I)\sin\varphi)\quad \text{ and }\quad \xi_{\oneM}(I,\varphi) = \xi_{\zeroM}(I,\varphi) + 2\pi.\label{eq:new_notation}
\end{equation}
This implies that the intersection point ``under'' $\sigma = \varphi$ is a point on the curve parameterized by $\xi_{\zeroM}(I,\varphi)$.
Otherwise, the intersection ``over'' $\sigma = \varphi $ is a point on the curve parameterized by $\xi_{\oneM}(I,\varphi)$.
As the slope of the {\NH} lines is $(I-1)/I$, given a point $(\theta,\theta)$, we obtain
\begin{equation*}
\frac{\xi_{\oneM}(I,\theta-I\tau_{\oneM}^*(I,\theta)) - \theta}{\theta - I\tau_{\oneM}^*(I,\theta) - \theta} = \frac{I-1}{I},
\end{equation*}
which can be rewritten as
\begin{equation*}
\frac{2\pi + \xi_{\zeroM}(I,\theta - I\tau_{\oneM}^*(I,\theta))- \theta}{- I\tau_{\oneM}^*(I,\theta)} = \frac{I-1}{I}.
\end{equation*}
From this equation, we obtain an expression for $\tau_{\oneM}^*(I,\theta$)
\begin{equation*}
\tau_{\oneM}^*(I,\theta) = \frac{-\left(2\pi + \xi_{\zeroM}(I,\theta - I\tau_{\oneM}^*(I,\theta)) - \theta\right)}{I-1}.
\end{equation*}
From the expressions of $\tau_{\oneM}^*(I,\theta)$ above and \eqref{eq:new_notation},
\begin{equation*}
\tau_{\oneM}^*(I,2\pi-\theta)= \frac{\left(\xi_{\zeroM}(I, \theta -I\left(- \tau_{\oneM}^*(I,2\pi -\theta)\right))+ \theta\right)}{I-1},
\end{equation*}
and therefore
\begin{equation*}
\theta - (I-1)(-\tau_{\oneM}^*(I,2\pi - \theta)) = \xi_{\zeroM}(I,\theta - (I-1)(-\tau_{\oneM}^*(I,2\pi-\theta))),
\end{equation*}
which implies that $-\tau_{\oneM}^*(I,2\pi- \theta)$ is a time of intersection between the {\NH} line and the curve parameterized by $\xi_{\zeroM}$.
In the case that there exists only one intersection point, this implies
\begin{equation*}
\tau_{\zeroM}^*(I,\theta) = \tau_{\oneM}^*(I,2\pi- \theta).
\end{equation*}
So, condition \eqref{eq:equi_lemma} is satisfied.
\end{proof}

\begin{proposition}
Let $\mathcal{S}_{\zerom}(I,\theta)$ be the scattering map associated to the graphs of $\xi_{\zerom}$ and $\eta_{\zerom}$ of $\mathcal{C}_1(I)$.
Assuming $a_1,a_2 >0$, for any $I$ there exists a $\theta_{+}$ such that $\dot{I}>0$ for $\theta \in (\pi, \theta_{+})$.
Moreover, $\theta_{+}\geq 3\pi/2$ for $I\notin(-1/2 , 1/2)$.  \label{prop:orb_cres}
\end{proposition}
\begin{proof}
A proof is given in Appendix~\ref{app:A}.
\end{proof}
\begin{remark}
If $a_1 < 0$, we have that there exists a $\theta_-$ such that $\dot{I}> 0$  for any $\theta \in (\theta_- , \pi)$.
\end{remark}
\begin{remark}
An analogous proposition holds for $\mathcal{S}_{\oneM}(I,\theta)$, the scattering map associated to the graphs of $\xi_{\oneM}$ and $\eta_{\oneM}$ of $\mathcal{C}_2(I)$.
In such case, there is a $\theta_+$ such that $\dot{I}\geq 0$ for any $\theta\in(\theta_+ , 2\pi)$ where $\theta \geq 3\pi/2$ for $I\in(1/2 , 3/2)$.\end{remark}

Note that this proposition leads us to ensure the diffusion in an analogous way to the one used to prove Theorem 4 in \cite{Delshams2017}.
Next, the diffusion mechanism is stated and the Arnold diffusion is proven.

\section{ Arnold Diffusion }
\label{sec:diffusion}

In this section we are going to complete our goal proving the existence of global instability or Arnold diffusion, that is, Theorem~\ref{theo:main_theo}.

We begin by presenting some general geometrical properties of the scattering maps that we have to take into account to prove the theorem of diffusion.
The first one reduces the study of scattering maps to positive values of $\mu$.
More precisely, we have the lemma below
%\paragraph{equivalencia mu positivo e negativo}
\begin{lemma}\label{lem:geometrical_lemmas}
The scattering map for a value of $\mu$ and $s = \pi$, associated to the intersection between $R(I,\varphi,s)$ and $C_{\m}(I)$ ($C_{\M}(I)$) has the same geometrical properties as the scattering map for $-\mu$ and $s = 0$, associated to the intersection between $R_{\theta}(I)$ and $C_{\M}(I)$ ($C_{\m}(I)$), i.e.,
\begin{equation*}\label{eq:equivalence_sign_mu}
S^{\mu}_{\text{m}(\M)}(I,\varphi,\pi) = S^{-\mu}_{\text{M}(\m)}(I,\varphi,0) = \mathcal{S}^{-\mu}_{\M(\m)}(I,\theta)
\end{equation*}
\end{lemma}
\begin{proof}
First, we look for $\tau^*_{\text{m}}$ such that the {\NH} segment $R(I,\varphi,s)$ intersects the crest $C_{\text{m}}(I)$.
If we fix $s=\pi$, we have from \eqref{eq:our_meln_potential} and \eqref{eq:L^*-def}:
\begin{equation}
L_{\mu,\m}^*(I,\varphi,\pi) = A_{1}(I)\cos(\varphi-I\tau_{\text{m}}^*(I,\varphi,\pi))+A_{2}(I)\cos(\varphi-\pi-(I-1)\tau_{\text{m}}^*(I,\varphi,\pi)).\label{eq:mel_pot_phi_mu_neg}
\end{equation}
Besides, $\tau^*$ satisfies
\begin{equation*}\label{eq:def_tau_mu_neg}
\mu\alpha(I)\sin(\varphi - I\tau_{\text{m}}^*) + \sin(\varphi -\pi-(I-1)\tau_{\text{m}}^*) =0,
\end{equation*}
or
\begin{equation*}
-\mu\alpha(I)\sin(\varphi - I\tau_{\text{m}}^*) +\sin(\varphi-(I-1)\tau_{\text{m}}^*) =0.\label{eq:equ_def_tau_mu_neg_and_pos}
\end{equation*}

We have that $\varphi - \pi -(I-1) \tau_{\text{m}}^* \pmod{2\pi} =  \xi_{\text{m}}(I,\varphi - I\tau_{\text{m}}^*)$ with $\pi/2 \leq\xi_{\text{m}}\leq 3\pi/2$.
Then, for each $\tau^*_{\m}$ there exists a $K\in\mathbb{Z}$ such that
$$\frac{\pi}{2} < \varphi-\pi - (I-1)\tau^*_{\m} + 2\pi K <\frac{3\pi}{2}.$$
This implies
$$\frac{3\pi}{2} < \varphi - (I-1)\tau^*_{\m} + 2\pi K \quad\text{ and }\quad \varphi - (I-1)\tau^*_{\m} + 2\pi (K-1) < \frac{\pi}{2}.$$
Therefore,
$$\varphi - (I - 1)\tau^*_{m} \pmod{2\pi} < \frac{\pi}{2}\quad\text{ or }\quad \varphi - (I - 1)\tau^*_{m} \pmod{2\pi} > \frac{3\pi}{2}.$$
We can conclude that $\varphi - (I - 1)\tau^*_{m} \pmod{2\pi} = \xi_{\M}(I,\varphi - I\tau^*_{\m})$.
Therefore $\tau_{\text{m}}^*(I,\varphi,\pi)$ for $\mu$ is equal to $\tau_{\text{M}}^*(I,\varphi,0)$ for $-\mu$.
From (\ref{eq:mel_pot_phi_mu_neg}), $L^*_{\mu,\m}(I,\varphi,\pi)$ satisfies
\begin{eqnarray*}
L_{\mu,\text{m}}^*(I,\varphi,\pi) &=& A_{1}(I)\cos(\varphi - \tau_{\text{M}}^*(I,\varphi,0)) +(-A_{2}(I))\cos(\varphi -(I-1)\tau_{\text{M}}^*(I,\varphi,0))\\
&=& L_{-\mu,\text{M}}^*(I,\varphi,0).
\end{eqnarray*}

Since $L^*_{\mu,\m}(\cdot,\cdot,\pi)$ and $L^*_{-\mu,\M}(\cdot,\cdot,0)$ coincide, their derivatives too and this implies that $$S^{\mu}_{\text{m}}(I,\varphi,\pi)=S^{-\mu}_{\text{M}}(I,\varphi,0) = \mathcal{S}^{-\mu}_{\text{M}}(I,\theta).$$
\end{proof}

From now on, just to simplify the exposition, $a_1$ and $a_2$ are considered positive.
The same strategy used in \cite{Delshams2017} is applied to prove the existence the diffusion:
we combine the scattering map in an interval of $\theta$ where $\dot{I}>0$ and the inner map to build a diffusion pseudo-orbit.
Then we apply shadowing results to get the existence of a diffusion orbit.

Since $I = 0$ and $I= 1$ are resonance values, the application of the inner map must be more careful, because in these resonance regions, for some orbits, the value of $I$ decreases in order $\mathcal{O}(\sqrt{\varepsilon})$, i. e., the tori cannot be considered flat.
We study the transversality between the foliations of invariant sets of the inner and the scattering map in resonant
and non-resonant regions and its image under the scattering map $\mathcal{S}$.
For more details and a more general case, the reader is referred to \cite{Delshams2009}.

Consider the resonant region associated to $I = 0$.
In such region, the tori can be approximated by $F^{0}(I,\varphi)$ given in \eqref{eq:invariant_tori_i=0}.
The tranversality between invariant sets of the inner and the scattering map holds if the gradient vectors of the level curves of $F^0$ and $\mathcal{L}^*$ are not parallel vectors, or equivalently,
$$\left\{F^0(I,\theta) , \mathcal{L}^*(I,\theta)\right\} \neq 0,$$
where $\left\{ , \right\}$ is the Poisson bracket,
$$\left\{F^0 , \mathcal{L^*}\right\} =  \frac{\partial F^0}{\partial \theta}\frac{\partial \mathcal{L}}{\partial I}-\frac{\partial F^0}{\partial I}\frac{\partial \mathcal{L}}{\partial \theta}.$$
From \eqref{eq:invariant_tori_i=0}, the partial derivatives of $F^0$ are
\begin{eqnarray*}
\frac{\partial F^0}{\partial I} = I &\text{  and  }&\frac{\partial F^0}{\partial \theta} = -\varepsilon a_1\sin\theta,
\end{eqnarray*}
and since $\mathcal{L}^*(I,\theta) = A_1(I)\cos(\theta - I\tau^*(I,\theta)) + A_2(I)\cos(\theta-(I-1)\tau^*(I,\theta))$, we have the partial derivatives given by
$$\frac{\partial \mathcal{L}^*}{\partial \theta} = \frac{A_1(I)\sin(\theta - I\tau^*)}{I-1},$$
$$\frac{\partial \mathcal{L}^*}{\partial I} = A_1'(I)\cos(\theta-I\tau^*) + A_2'(I)\cos(\theta- (I-1)\tau^*) + A_1(I)\tau^* \sin(\theta - I\tau^*) + A_2(I)\tau^*\sin(\theta - (I-1)\tau^*).$$

Note that if $\left|I\right|> \mathcal{O}(\varepsilon)$, $\partial F^0/\partial I$ dominates $\partial F^0/\partial \theta$, so the Poisson bracket above can be reduced to
$$\left\{F^0 , \mathcal{L^*}\right\} \simeq -\frac{\partial F^0}{\partial I}\frac{\partial \mathcal{L}}{\partial \theta} = \frac{-IA_1(I)\sin(\theta - I\tau^*)}{I-1}$$

Expanding $\sin(\theta - I\tau^*)$ in Taylor's series around $I = 0$, we have
$$\sin(\theta - I\tau^*) = \sin\theta + \mathcal{O}(I),$$
which implies $\left\{F^0,\mathcal{L}^*\right\}=0$ if, and only if, $\theta \approx 0,\pi$, assuming that $\mathcal{O}(I)$ is small enough.

Now, we consider $I = \mathcal{O}(\varepsilon)$ and look at the intersections between the {\NH} lines and the graph of $\xi_1$.
Note that as the value of $I$ is close to $0$ we can assume that the crests are horizontal.
Using Taylor's series we can write
\begin{eqnarray*}
\sin(\theta - I\tau^*) = \sin\theta + \mathcal{O}(I) && \cos(\theta - I\tau^*) = \cos\theta +  \mathcal{O}(I)\\
\sin(\theta - (I-1)\tau^*) = \mathcal{O}(I) &&\cos(\theta - (I-1)\tau^*) = -1 +\mathcal{O}(I).
\end{eqnarray*}
This implies
\begin{equation}
\left\{F^0 , \mathcal{L}^*\right\}= -\frac{IA_1(I)\sin\theta}{I-1} -\varepsilon a_1\sin\theta\left(A_1'(I)\cos\theta - A_2'(I) + A_1(I)\tau^*\sin\theta \right) + \mathcal{O}(I^2 ,\varepsilon I).\label{eq:transv_pois_b}
\end{equation}

Taylor expanding the functions $A_1(I)$, $A_1'(I)$ and $A_2'(I)$ around $I = 0$, we obtain
\begin{equation*}
A_1(I) = 4a_1 + \mathcal{O}(I²), \quad A_1'(I) = \mathcal{O}(I) \text{ and }  A_2'(I) = a_2\pi(\pi\coth\pi/2 - 2)\text{csch}(\pi/2) + \mathcal{O}(I)
\end{equation*}
Plugging these expressions in \eqref{eq:transv_pois_b}, we set
$$\left\{F^0 , \mathcal{L^*}\right\} =  -\frac{4a_1I\sin\theta}{I-1} -\varepsilon a_1\sin\theta\left[a_2\pi(\pi\coth\pi/2 - 2)\text{csch}(\pi/2) + 4a_1(\pi - \theta)\sin\theta\right]  + \mathcal{O}(I^2 , I\varepsilon).$$
So, $\left\{F^0 , \mathcal{L^*}\right\} = 0 \Leftrightarrow  a_1\sin\theta\left[\frac{-4I}{I-1}-\varepsilon a_2\pi (\pi\coth\pi/2 - 2)\text{csch}(\pi/2) + \varepsilon4(\pi - \theta)\sin\theta\right]= 0$.
In other words, we do not have transversality if, and only if, $\theta = 0,\pi$ or satisfies
$$(\pi - \theta)\sin\theta = \frac{I}{\varepsilon a_1} + \frac{\pi(\coth\pi/2 - 2)\text{csch}\pi/2)}{4},$$
which is not an horizontal curve in the plane $(\theta , I)$ and is transversal to an invariant torus of the inner dynamics.
%############### transversalidade versao 2 - final

For the other resonant region $I = 1$, $F^1$ is very similar.
Assuming $I-1 = \mathcal{O}(\varepsilon)$, we have
\begin{equation*}
\left\{F^1 , \mathcal{L}^*\right\} = a_2\sin\theta \left\{4\left(\frac{I-1}{I}\right) - \varepsilon\left[\pi a_1(2-\pi\coth(\pi/2))\text{csch}(\pi/2) + 4 a_2 \sin\theta \right]\right\}.
\end{equation*}

Applying the same methodology, we obtain an analogous result for the other resonant region $F^1$.
In short, we conclude that the image $\mathcal{S}(\mathcal{T}_i)$ of an invariant torus $\mathcal{T}_i$ of
the inner map under the scattering map intersects tranversally another invariant torus $\mathcal{T}_{i+1}$
of the inner map.

Finally, in the non-resonant region, we notice that
$$\left\{F^{\text{nr}} , \mathcal{L}^*\right\} = -\frac{\partial F^{\text{nr}}}{\partial I}\frac{\partial \mathcal{L}^*}{\partial \theta} = -\frac{IA_{1}(I)\sin(\theta - I\tau^*)}{I-1},$$
just the same expression as the one for the resonance $I = 0$, so the transversality between invariant sets of the inner and the scattering map follows.

Now, a constructive proof of Theorem~\ref{theo:main_theo} is presented.
This proof is similar to the proof presented in \cite{Delshams2017}, but now, there is no any piece of ``highway'' or fast vertical lines where $\left|I\right|$ is large.
So, the inner map is applied more times.

\subsection{Proof of Theorem~\ref{theo:main_theo}}

\begin{proof}
First of all we have to choose what scattering map we use.
This choice depends on the sign of $\mu$ as explained in Lemma~\ref{lem:geometrical_lemmas}.
Assuming $\mu >0$, we take $\mathcal{S}_{\zerom}(I,\theta)$, the global scattering map associated to the graphs of $\xi_{\zerom}$ and $\eta_{\zerom}$.
If $a_1 >0$, by Proposition~\ref{prop:orb_cres} for any $I$ there exists an interval $\theta \in (\pi,\theta_+)$ where $\dot{I}>0$.
Define $H_{\text{r}}$ the set $\left(\rho,\theta_+\right)\times\left[-I^*,I^*\right]$, where $\rho = \pi + \delta$ is such that $\pi <\rho<\theta_+$ and  the transversality between {\NH} lines and $\mathcal{L}_1^*$ holds.
We first construct a pseudo-orbit $\{(I_i,\theta_i): i = 0,\dots,N_1\}\subset H_{\text{r}}$ with $I_0 = -I^*$  and $\theta_{N_1}$ as close as possible to $\rho$.
Note that all these points lie in the same level curve of $\mathcal{L}_{\zerom}^*$, that is, $\mathcal{L}_{\zerom}^*(I_0,\theta_0) = \mathcal{L}_{\zerom}^*(I_i,\theta_i)$, $i = 1,\dots,N_1$.
Applying the inner dynamics, we get $(I_{\text{N}_1 + 1},\theta_{\text{N}_1 + 1}) = \phi_{t_{N_1}}(I_{\text{N}_1},\theta_{\text{N}_1} )$ with $\theta_{\text{N}_1 + 1}\in (\rho , \theta_+)$ and then we construct a pseudo-orbit $\{(I_i,\theta_i): i = N_1 +1,\dots,N_1+M_1\}\subset \mathcal{L}_{\zerom}^*(I_{\text{N}_1+1},\theta_{\text{N}_1+1}) = l_{\text{N}_1 +1}$ with $\theta_{i}\in(\rho,\theta_{\text{N}_1 +1})$, $\theta_+ -\theta_{\text{N}_1 + \M_{1}} = \mathcal{O}(\varepsilon^2)$.
Applying the inner dynamics, we get  $(I_{\text{N}_1 + \M_1+1}, \theta_{\text{N}_1 + \M_1+1})=\phi_{t_{\text{N}_1 + \M_1}}(I_{\text{N}_1 + \M_1}, \theta_{\text{N}_1 + \M_1})$ with $\theta_{\text{N}_1 + \M_1 +1}\in (\rho,\theta_+)$.
Recursively, we construct a pseudo-orbit $\{(I_i,\theta_i): i = \text{N}_1 +1,\dots,\text{N}_2\}$ such that $I_{\text{N}_2}\geq I^*$.
In the same ways as in \cite{Delshams2017} (Theorem 4), we can apply shadowing techniques of \cite{fontich2000,fontich2003,Gidea2014}, due to the fact that the inner dynamics is simple enough to satisfy the required hypothesis of these references, to prove the existence of a diffusion trajectory.
If $a_{10}<0$, changing $H_{\text{r}}$ to $H_{\text{l}}=\left(\theta_+,\pi\right)$ all the previous reasoning applies.
\end{proof}

\begin{remark}Considering  Remark~\ref{rem:r_1}, Remark~\ref{rem:r_2}, Remark~\ref{rem:r_3} and  Remark~\ref{rem:r_4}, for any $r\in(0,1)$, an equivalent diffusion result is readily obtained.
\end{remark}

\section{Piecewise smooth global scattering maps}
\label{sec:piecewise}

In this section, the geometric freedom of the choice of $\tau^*$ is explored.
Until now, only two different scattering maps have been used to build a global one, and this was enough to ensure diffusion.
But, with this approach, finding a diffusion pseudo-orbit is not always easy enough and this pseudo-orbit can be also complicated.
This depends simply on the ``aspect" of the scattering map obtained.

We now suggest a new criterion to choose $\tau^*$: to take the minimal value for $\left|\tau^*\right|$ for any $(\theta , I)$.
This provides us with a piecewise smooth global scattering map with a good property: the phase space of this scattering map which is $\mathcal{O}(\varepsilon^2)$-close to the level sets of the reduced Poincar\'{e} function $\mathcal{L}^*(I,\theta)$ associated to the chosen $\tau^*$ is simpler and ``cleaner'' than the phase spaces of other scattering maps displayed up to now.
By a cleaner scattering map, we mean that we can easily identify and understand the orbits of the scattering maps, except for a small region which contains the tangency locus.

%========================= fig: piecewise smooth scattering maps =========
\begin{figure}[h]
\centering
\subfigure[Piecewise scattering map for $\mu = 0.3$.]{\includegraphics[scale=0.29]{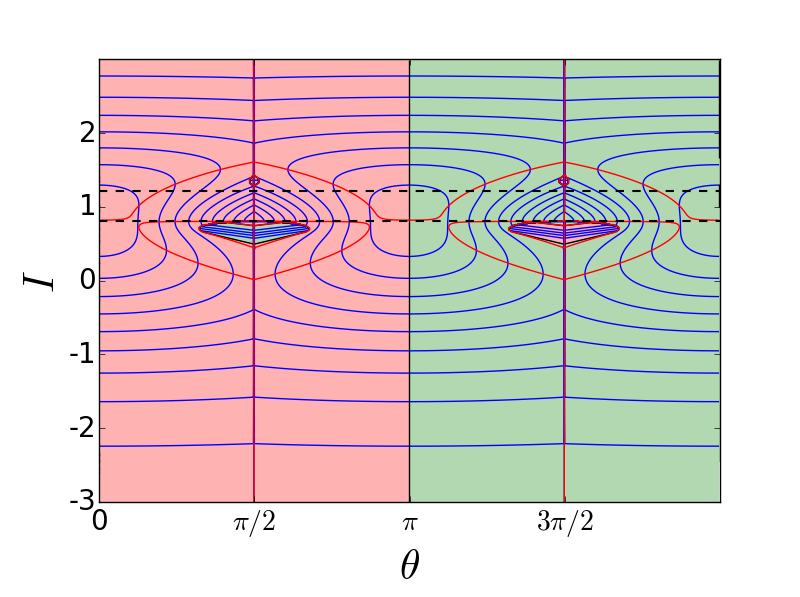}\label{fig:pw_03}}
\subfigure[Piecewise scattering map for $\mu = 0.5$.]{\includegraphics[scale=0.27]{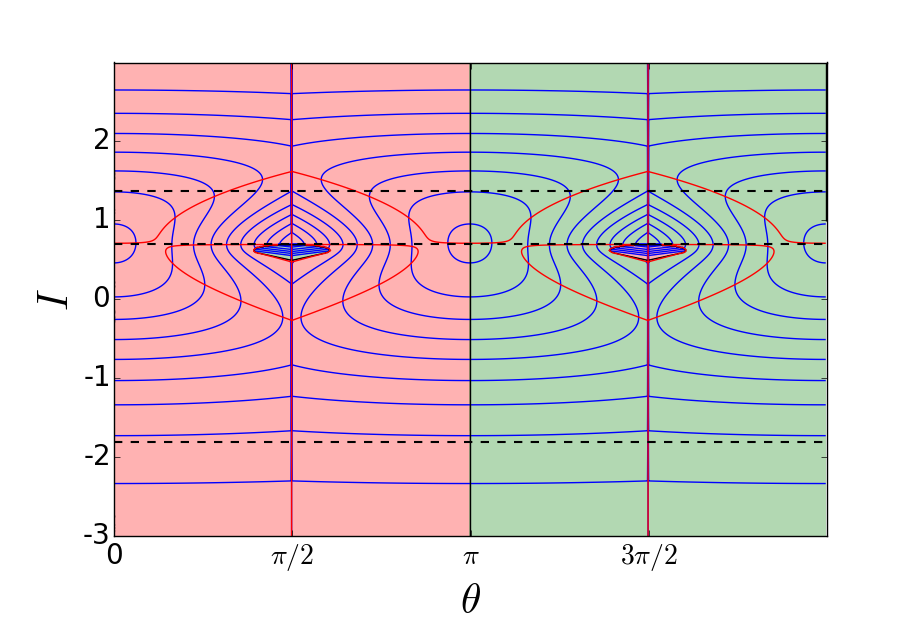}\label{fig:pw_05}}
\subfigure[Piecewise scattering map for $\mu = 0.9$.]{\includegraphics[scale=0.27]{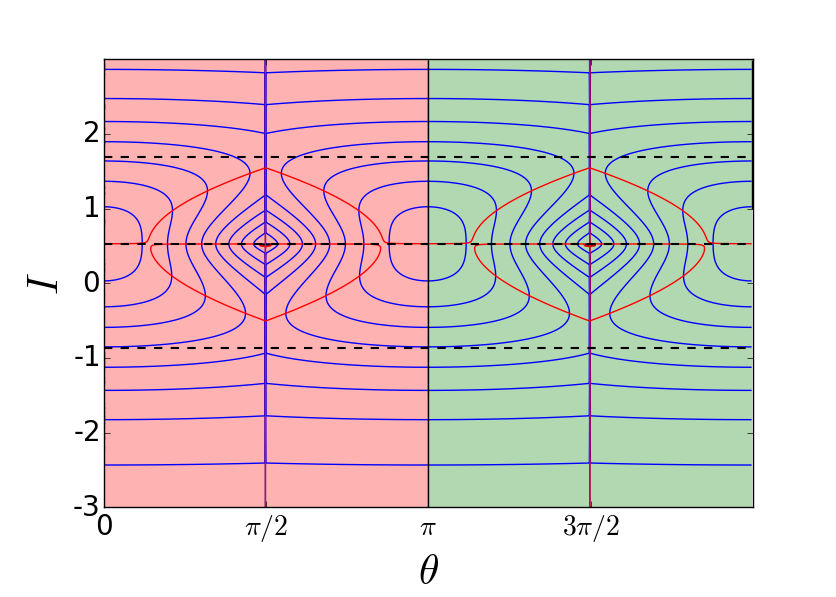}\label{fig:pw_09}}
\subfigure[Piecewise scattering map for $\mu = 1.5$.]{\includegraphics[scale=0.27]{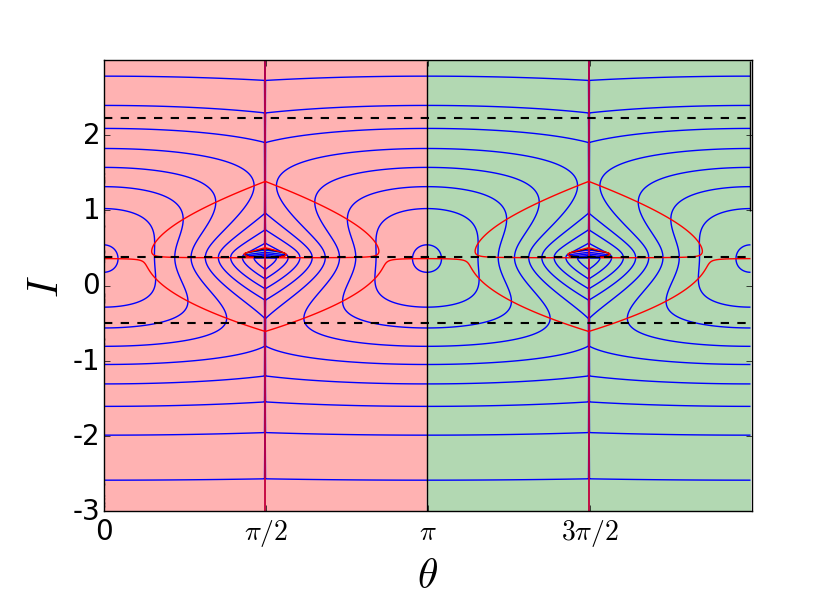}\label{fig:pw_15}}
\caption{Examples of piecewise smooth global scattering maps. The orbits of scattering maps are represented by the blue lines.
In the red zones the values of $I$ on such orbits decrease, in the green one the values of $I$ increase. \label{fig:global_sm_piece}}
\end{figure}
%=============================================================================

Besides, the zones where the value of $I$ is increased or decreased under the scattering map is well behaved.
$I$ decreases for $\theta \in (0,\pi)$ (the red region on all pictures in Fig.~\ref{fig:global_sm_piece}) and $I$ increases for $\theta  \in (\pi , 2\pi)$ (the green region on all pictures in Fig.~\ref{fig:global_sm_piece}).
So it is easy to infer that for finding a diffusion pseudo-orbit it is enough to build a combination between the inner map and this scattering map restricted to $(\pi , 2\pi)$, for example if an increased value of $I$ is wished.
The same idea used in the proof of Theorem~\ref{theo:main_theo}.

Observe that the scattering maps we are now considering are a mix of the scattering maps studied previously.
As an example, we illustrate the scattering map obtained for $\mu = 0.9$.
Such scattering map can be divided into three regions and in each region, the scattering map coincides with a scattering map studied before.

%========================= fig: regions ====================
\begin{figure}[h]
\centering
\includegraphics[scale=0.3]{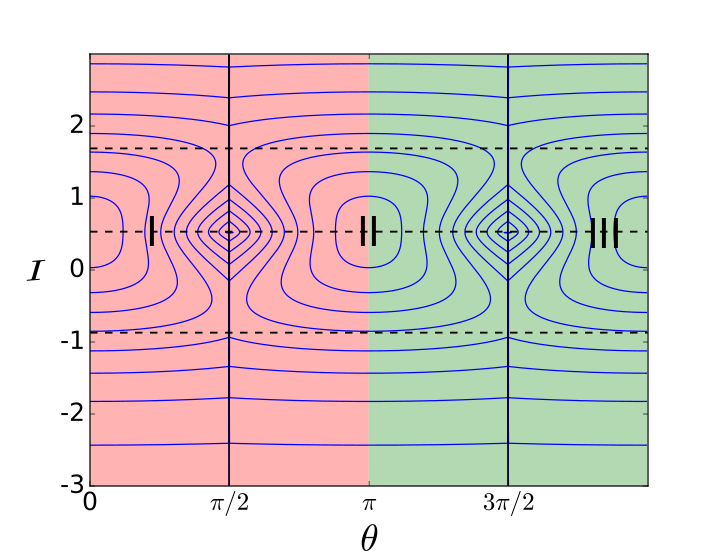}
\caption{A piecewise smooth global scattering map divided into 3 regions. The vertical black lines are the boundaries of
the domains of smooth scattering maps.\label{fig:3_regions}}
\end{figure}
%=========================================================

In Fig.~\ref{fig:3_regions}, for regions I ($0<\theta<\pi/2$), II ($\pi/2<\theta<3\pi/2$) and III ($3\pi/2<\theta<2\pi$) the scattering map has the following correspondence:
\begin{itemize}
\item[I] Extended scattering map $\mathcal{S}_0(I,\theta)$ associated to the horizontal $\mathcal{C}_{\M}(I)$ ``under" $\sigma = \varphi$.
\item[II] Extended scattering map $\mathcal{S}_1(I,\theta)$ associated to the horizontal $\mathcal{C}_{\m}(I)$.
\item[III] Extended scattering map $\mathcal{S}_{2}(I,\theta)$ associated to the horizontal $\mathcal{C}_{\M}(I)$ ``over" $\sigma = \varphi$.
\end{itemize}

If extended scattering maps are not considered and we just use scattering maps associated to horizontal and vertical crests, one can see that these scattering maps can be divided into 6 regions, i.e., they can be viewed as a combination of up to 6 scattering maps.

Another property of these scattering maps is the loss of differentiability on the straight lines $\theta = \pi/2$ and $\theta = 3\pi/2$.
The vector field associated to the Hamiltonian $-\mathcal{L}_i^*$ defined around these discontinuity lines behaves as the vector fields studied in non-smooth dynamics theory.
More precisely, we can find regions with slide and unstable slide behavior \cite{Filippov88}.
In a future work, we envisage to design special pseudo-orbits along these discontinuity lines using such theory.
Note that these pseudo-orbits would be very similar to the ``highways" defined in \cite{Delshams2017}, so in principle, one can expect fast and simple diffusion along these discontinuity lines.

\subsection*{Acknowledgments}

The authors would like to express their gratitude to the
anonymous referees for their comments and suggestions which
have contributed to improved the final form of this paper.
We also thank C. Sim\'o for several discussions and comments.

\appendix
\section{Proof of Proposition~\ref{prop:orb_cres} \label{app:A}}
\begin{repproposition}{prop:orb_cres}
Let $\mathcal{S}_{ \zerom}(I,\theta)$ be the scattering map associated to the graphs $\xi_{\zerom}$ and $\eta_{\zerom}$.
Assuming $a_1,a_2 >0$, then for any $I$, there exists a $\theta_{+}$ such that $\dot{I}>0$ for $\theta \in (\pi, \theta_{+})$.
Moreover, $\theta_{+}\geq 3\pi/2$ for $I\notin(-1/2 , 1/2)$.  \end{repproposition}
\begin{proof}
We have
\begin{equation}
\dot{I} =\frac{\partial \mathcal{L}^*}{\partial\theta}(I,\theta) = \frac{A_1(I)\sin(\theta - I\tau^*(I,\theta))}{I-1}  = -\frac{A_2(I)\sin(\theta - (I-1)\tau^*(I,\theta)}{I}.\label{eq:dot_I}
\end{equation}
where $A_{1}(I)$ and $A_{2}(I)$ are positive, because $a_1,\,a_{2}>0$.
Notice that $\mu = a_1/a_2 >0$.

Note that as $(I , \varphi = \pi,\theta = \pi)$ is always on the crest $\mathcal{C}_{\m}(I)$, $\tau^*(I,\pi)= 0 $ for all $I$.

Consider first the case of horizontal crests ($\left|\alpha(I)\mu\right|< 1$).
\begin{itemize}
\item[a)] For $I< 0$, the function $\alpha(I)$ introduced in \eqref{eq:mu_alpha} satisfies $\alpha(I)> 0 $, and from \eqref{eq:cristas_06}, $\sin(\xi_{\zerom}(I,\varphi))\sin\varphi = -\mu \alpha(I)\sin\varphi\leq 0$.
Take $\theta = \frac{3\pi}{2}$; since $I<0$, the slope $m = (I-1)/I$ of the {\NH} lines is greater than 1.
Therefore, $3\pi/2 - I\tau_{\zerom}^*(I,3\pi/2)\in(\pi,3\pi/2)$.
This implies that for any $\theta \in (\pi, 3\pi/2)$, $\theta - I\tau_{\zerom}^*(I,\theta) \in (\pi,3\pi/2)$, so $\sin(\theta - I\tau_{\zerom}^*)<0$.
From \eqref{eq:dot_I}, $\dot{I}>0$.
\item[b)] For $0<I<1$, $\alpha(I)< 0$, so $\sin\xi_{\zerom}(I,\varphi)\sin\varphi\geq 0$.
Besides, $m<0$, so if we look for $\theta_*$ satisfying
\begin{eqnarray}
\theta - I\tau = 2\pi\label{eq:lem_sys_1}\\
\theta - (I-1)\tau = \pi,\nonumber
\end{eqnarray}
we have that for any $\theta\in(\pi,\theta_*)$, $\theta - I\tau^*_{\zerom}\in(\pi,2\pi)$.
By solving \eqref{eq:lem_sys_1} and defining $\theta_+ := \theta_*$, we obtain $\theta_+ = (2-I)\pi$.
Then, $\sin(\theta - I\tau^*_{\zerom}(I,\theta))<0$ and therefore $\dot{I}>0$ for any $\theta \in (\pi,\theta_+ = (2-I)\pi)$.
In particular, $\theta_+ <3\pi/2$ if, and only if, $I\in(1/2,1)$.
\item[c)] For $I >1$, one more time $\alpha(I) > 0$ and $\sin\xi_{\zerom}(I,\varphi)\sin(\varphi)< 0$, but now $ 0 < m = 1 - 1/I < 1 $.
We first fix $ \theta = 3 \pi/ 2 $ and search for $I$ such that
\begin{eqnarray*}
\frac{3\pi}{2} -I\tau^*(I,3\pi/2) = 0\\
\frac{3\pi}{2} - (I - 1)\tau^*(I,3\pi/2) = \pi.
\end{eqnarray*}
We obtain $I = 3/2$, so  $\theta - I\tau^*_{\zerom}(I,\theta) \in (0,\pi)$ for any $I \geq 3/2$ and $\theta \in (\pi , \theta_+ = 3\pi/2)$.
Consequently, $\sin(\theta - I\tau^*_{\zerom}(I,\theta))>0$ and $\dot{I}>0$.
For the values of $I\in(1,3/2)$ we change the strategy.
We look for $\theta_*$ such that
\begin{eqnarray*}
\theta - I\tau^*(I,\theta) = 0\\
\theta - (I-1)\tau^*(I,\theta) = \pi.
\end{eqnarray*}
We have $\theta_* = \pi I$ and $\theta - I\tau^*_{\zerom}(I,\theta_*)\in(0,\pi)$ for any $I\in(1,3/2)$ and $\theta\in(\pi, \theta_*)$, so $\dot{I}>0$.
Note that $\theta_* < 3\pi/2$ and we can define $\theta_+ := \theta_*$.
\end{itemize}
Observe that for $I = 1 $ the crests are vertical, and for $I=0$, $\theta = \theta - I\tau^*_{\zerom}(I,\theta)$, and $\dot{I}>0$ for $\theta \in (\pi,3\pi/2)$.

Consider now the case of vertical crests ($\left|\alpha(I)\mu\right|>1$).
\begin{itemize}
\item[a)]For $I< 0$, $\sin\eta_{\zerom}(I,\sigma)\sin\sigma = -\mu\alpha(I)\sin^2\sigma \leq 0 $ and $m>1$.
We fix $\theta = 3\pi/2$ and look for $I$ such that
\begin{align*}
3\pi/2\pi - I\tau^* &= \pi\\
3\pi/2 - (I-1)\tau^*(I,3\pi/2) &= 0.
\end{align*}
We obtain $I = -1/2$ and therefore, $\sin(\theta - (I-1)\tau^*_{\zerom}(I,\theta))> 0$ for $I\in(-\infty,-1/2)$ and $\theta \in(\pi,3\pi/2)$.
Consequently, $\dot{I} > 0$ from \eqref{eq:dot_I}.
For $I\in(-1/2 , 0)$, we have that $\theta_+ = (1-I)\pi$ satisfies
\begin{align*}
\theta - I\tau^*(I,\theta_+) & = \pi\\
\theta_+ -(I-1)\tau^*(I,\theta_+) & = 0.
\end{align*}
Therefore, $\sin(\theta - (I-1)\tau^*_{\zerom})(I,\theta) >0$ and $\dot{I}>0$ for any $\theta \in  (\pi,\theta_+)$.
\item[b)] For $ 0 < I < 1$ $\sin\eta_{\zerom}(I,\sigma)\sin\sigma \geq 0 $ and $m<0$.
$\theta_+ = (I + 1)\pi$ satisfies
\begin{align*}
\theta - I\tau^*(I,\theta_+) &= \pi\\
\theta_+ -(I-1)\tau^*(I,\theta_+)& = 2\pi.
\end{align*}
So, $\sin(\theta - (I-1)\tau^*_{\zerom}(I,\theta))>0$ and $\dot{I}>0$ for any $\theta \in(\pi,\theta_+)$.
Note that $\theta_{+} <3\pi/2$ for $I \in (0,1/2)$.
\item[c)] Finally, for $I > 1$, $\sin\eta_{\zerom}(I,\sigma) \sin\sigma \leq 0$.
We have that $\theta - (I-1)\tau^*_{\zerom}(I,\theta)\in(\pi,2\pi$), so $\sin(\theta - (I-1)\tau^*_{\zerom}(I,\theta))<0$  and $\dot{I}> 0$ for any $\theta\in(\pi,3\pi/2)$.
\end{itemize}
For $I = 0$ the crests are horizontal. For $I = 1$, $\theta = \theta - (I-1)\tau^*_{\zerom}(I,\theta)$, so $\dot{I}>0$ for $\theta \in(\pi,2\pi)$.

\end{proof}

\bibliography{references}

\end{document}